\documentclass[a4paper,oneside,10pt]{article}

\usepackage[a4paper,width=150mm,top=25mm,bottom=25mm,bindingoffset=0mm]{geometry}

\usepackage[english]{babel}

\usepackage{graphicx}
\usepackage[font=small,labelfont=bf]{caption}
\usepackage{subcaption}
\usepackage{color}
\usepackage{mdframed}
\usepackage{framed} 
\usepackage{cancel}
\usepackage[normalem]{ulem}

\usepackage{multicol}

\usepackage{tikz}
\usetikzlibrary{tikzmark}

\newcommand{\pkg}[1]{\textsf{#1}}

\usepackage{bookmark}
\usepackage{cancel}
\usepackage{multirow}
\usepackage{hhline}
\usepackage{diagbox}
\usepackage{makecell}
\usepackage{slashbox,booktabs}

\usepackage{url}

\usepackage{enumerate}
\usepackage{here}

\usepackage[normalem]{ulem}
\usepackage{ragged2e}
\usepackage{csquotes}
\usepackage{tasks}
\usepackage{paralist}
\usepackage[inline]{enumitem}
\usepackage{tikz-cd}
\usepackage{stmaryrd}

\usepackage{lineno,hyperref}
\modulolinenumbers[5]

\usepackage{amsmath,amssymb,amsthm,mathtools,mathrsfs,epsfig,amsfonts,wasysym}
\usepackage{MnSymbol}

\usepackage[normalem]{ulem}
\usepackage{ragged2e}
\usepackage{csquotes}
\usepackage{tasks}
\usepackage{paralist}
\usepackage[inline]{enumitem}
\usepackage{tikz-cd}
\usepackage{stmaryrd}

\usepackage[titletoc,title]{appendix}

\newcommand{\e}{\mathrm{e}}
\newcommand{\de}{\mathrm{d}}

\usepackage{url,hyperref}
\hypersetup{colorlinks=true,linkcolor=blue,filecolor=magenta,urlcolor=cyan}
\usepackage{nicematrix}
\usepackage{tikz}
\usetikzlibrary{fit}

\usepackage[all]{xy}

\usepackage{booktabs}
\usepackage{tabularray}
\UseTblrLibrary{diagbox}
\usepackage{colortbl}

\newtheorem{corollary}{Corollary}

\newtheorem{theorem}{Theorem}
\newtheorem{proposition}{Proposition}

\usepackage{pifont}

\usepackage{authblk}
\usepackage{todonotes}

\definecolor{wildstrawberry}{rgb}{1.0, 0.26, 0.64}

\definecolor{ao(english)}{rgb}{0.0, 0.5, 0.0}

\usepackage{orcidlink}

\begin{document}

\title{An age-structured SVEAIR epidemiological model} 

\author[1]{Vasiliki Bitsouni\,\orcidlink{0000-0002-0684-0583}\thanks{\texttt{vbitsouni@math.upatras.gr}}}

\author[2,3]{Nikolaos Gialelis\,\orcidlink{0000-0002-6465-7242}\thanks{\texttt{ngialelis@math.uoa.gr}}}

\author[1]{Vasilis Tsilidis\,\orcidlink{0000-0001-5868-4984}\thanks{\texttt{vtsilidis@upatras.gr}}}

\affil[1]{Department of Mathematics, University of Patras, GR-26504 Rio Patras, Greece}

\affil[2]{Department of Mathematics, National and Kapodistrian University of Athens, GR-15784 Athens, Greece}

\affil[3]{School of Medicine, National and Kapodistrian University of Athens,\newline GR-11527 Athens, Greece}

\date{}

\maketitle

\begin{abstract}
\noindent
In this paper, we introduce and study an age-structured epidemiological compartment model and its respective problem, applied but not limited to the COVID-19 pandemic, in order to investigate the role of the age of the individuals in the evolution of epidemiological phenomena. We investigate the well-posedness of the model, as well as the global dynamics of it in the sense of basic reproduction number via constructing Lyapunov functions. 
\end{abstract}




\section{Introduction}
\label{intro}

Epidemiological mathematical models have played a crucial role in understanding and predicting the spread of infectious diseases, as well as informing public health policies and measures (see \cite{ianpu,iannelli2017basic,sharpe1911problem,m1925applications,xiang2021covid} and many references therein). With the emergence of the COVID-19 pandemic, the importance of these models has been highlighted as they have been used to assess the anticipated spread of the virus and inform strategies to mitigate its impact \cite{adiga2020mathematical,jewell2020predictive,watson2022global}.

One of the key aspects of modern epidemiological models is the incorporation of age-structured models, which take into account the differences in susceptibility, transmission, and disease progression across various age groups (such as \cite{iannelli2017basic,inaba1990threshold}). This approach allows for a more accurate representation of disease dynamics and enables better targeting of interventions and resource allocation.

The aim of the present paper is to investigate the role of the age of the individuals in the evolution of epidemiological phenomena. Using as a case study the COVID-19 outbreak we aim to address the following questions: 
\begin{itemize}
\item[--] How does the age of individuals affect the spread of the epidemic? 
\item[--] What is the effect of the asymptomatic infectious individuals on the basic reproduction number, $\mathcal{R}_0$, of COVID-19?
\end{itemize}

We answer the above questions  by deriving an age-structured epidemiological compartment model  that incorporates the important role of both asymptomatic and symptomatic individuals.

This study is organized as follows.  In \hyperref[sec:model]{\S \ref*{sec:model}}, we develop a novel age-structured SVEAIR model that incorporates, among others, the ambiguous (see \hyperref[sec:derivanal]{\S \ref*{sec:derivanal}}) variable of asymptomaticity of infectious individuals for the spread of COVID-19 disease. We show its global well-posedness, we derive the basic reproductive number, $\mathcal{R}_0$, of the model and we study the global stability of its steady states. In \hyperref[sec:numerics]{\S \ref*{sec:numerics}}, we undertake numerical simulations to confirm the behaviour of the solution of the problem.  We conclude in  \hyperref[sec:CD]{\S \ref*{sec:CD}} with a summary and discussion of the results. 

\section{The epidemiological model}
\label{sec:model} 

Here we introduce an epidemiological model, $\mathscr{M}$, along with the respective problem, $\mathscr{P}$, as a means of utilization of the proposed scheme in answering the main question of the present paper. 

\subsection{Derivation and analysis of the model}
\label{sec:derivanal} 

One of the most critical facts about COVID-19, is that a significant number of cases, mainly those of young age, has been reported as asymptomatic (see \cite{gao2021systematic} and many references therein), leading to fast spread of the infection. Although the asymptomatic cases have a shorter duration of viral shedding and lower viral load \cite{li2020substantial,Yang2020}, their proportion can range from 4\%-90\% (see \cite{heneghancovid,sah2021asymptomatic} and many references therein) and most of the time they play a key role in infection transmission. Therefore, we incorporate not only both symptomatic and asymptomatic cases in our model (as it is done in, e.g., \cite{bitsgialstrcovid2021}), but also the age of the infected/infectious individuals. 

In particular, the proposed $\mathscr{M}$ is based on the following hypotheses. 
\begin{enumerate}
\item The \textit{total population}, $N$, is classified into six non-negative-valued compartments, \textit{susceptible}, $S$, \textit{vaccinated-with-a-prophylactic-vaccine}, $V$, \textit{latent/exposed}, $E$, \textit{asymptomatic infectious}, $A$, \textit{symptomatic infectious}, $I$, and \textit{recovered/removed}, $R$, \textit{individuals}, thus $$N=S+V+E+A+I+R.$$ All of the above epidemiological variables depend on non-negative \textit{time}, $t$. 
\item 
\begin{enumerate}[label=\roman*.]
\item There is also another independent non-negative \textquote{\textit{age}}-variable, $\theta$, which measures the time elapsed since, e.g., birth or infection. The two time-variables have different scales, i.e they are measured in different units, and the parameter $\omega\in\mathbb{R}^+$ stands for the \textit{conversion factor} from the units of $\theta$ to the units of $t$. 
\item Only the non-negative-valued \textit{age-densities} of $E$, $A$ and $I$, i.e. $e$, $a$ and $i$, respectively, contribute to our $\mathscr{M}$. Those densities should vanish at (or have already vanished before) $\theta\to\infty$, hence it is natural for them to be considered as elements of $L^1{\left(\mathbb{R}_0^+\right)}$, for every fixed $t$. In the light of the above assumption, the expressions $$E=\int\limits_0^\infty{e{\left(\,\cdot\,,\theta\right)}\,\de\theta},\text{ }A=\int\limits_0^\infty{a{\left(\,\cdot\,,\theta\right)}\,\de\theta}\text{ and }I=\int\limits_0^\infty{i{\left(\,\cdot\,,\theta\right)}\,\de\theta}$$ are well-posed. 
\end{enumerate}
\item 
\begin{enumerate}[label=\roman*.] 
\item The vaccine is considered to be \textit{purely prophylactic}. 
\item Only a part of population is vaccinated and $p\in\left[0,1\right]$ stands for the \textit{vaccine coverage}. Since the vaccine is supposed to be purely prophylactic, the only source for the vaccinees concerns the pool of the susceptible individuals. That source is considered to be linear. 
\item The vaccine is likely to be imperfect (at providing prophylaxis) and $\epsilon\in\left[0,1\right]$ stands for its \textit{effectiveness}. 
\item The \textit{vaccine-induced immunity}, i.e. the process of vaccinees obtaining immunity and moving into recovered population, is considered to be linear and the letter $\zeta\in\mathbb{R}_0^+$ is employed for the vaccine-induced immunity rate. 
\end{enumerate}
\item The \textit{transmission}, i.e. the process of susceptible individuals and failed-to-be-immune vaccinees becoming latent, is considered to be exclusively \textit{horizontal} and to be governed by the \textit{Holling-type-I functional response}. The parameters $\beta_A,\beta_I\,\in L^\infty{\left(\mathbb{R}_0^+;\mathbb{R}_0^+\right)}$ stand for the transmission rates of asymptomatic and symptomatic, respectively, infectious individuals. 
\item The \textit{incubation}, i.e. the process of latent individuals becoming infectious, is considered to be linear and $k\in L^\infty{\left(\mathbb{R}_0^+;\mathbb{R}_0^+\right)}$ is the incubation rate. That rate is the same for both asymptomatic and symptomatic classes, but those classes are different from each other in terms of magnitude of their sources. In particular, $q\in L^\infty{\left(\mathbb{R}_0^+;\left[0,1\right]\right)}$ stands for the proportion of latent individuals that become asymptomatic infectious ones. 
\item The \textit{recovery}, i.e. the process of infectious individuals moving into the recovered population, is considered to be linear and $\gamma_A,\gamma_I\,\in L^\infty{\left(\mathbb{R}_0^+;\mathbb{R}_0^+\right)}$ stand for the recovery rates of asymptomatic and symptomatic, respectively, infectious ones. 
\item 
\begin{enumerate}[label=\roman*.]
\item Some of the asymptomatic infectious individuals never develop symptoms and they move directly into the recovered/removed class and the letter $\xi\in L^\infty{\left(\mathbb{R}_0^+;\left[0,1\right]\right)}$ is employed for the proportion of those asymptomatic infectious individuals. 
\item The \textit{symptomatic transition}, i.e. the process of asymptomatic infectious individuals turning into symptomatic ones, is considered to be linear and $\chi\in L^\infty{\left(\mathbb{R}_0^+;\mathbb{R}_0^+\right)}$ stands for the symptomatic transition rate. 
\end{enumerate}
\item \textit{Demographic terms} are taken into account and they are considered to be linear, with $\mu\in\mathbb{R}^+$ being the universal birth/death rate. We note that $\mu$ is considered to be the only strictly positive constant of $\mathscr{M}$.
\item No \textit{reinfections} are taken into account, hence no movement from the pool of the removed individuals to the pool of the susceptible ones is considered. 
\end{enumerate} 

The respective \textit{initial-boundary value $\mathscr{P}$} has the following form: For given $$\left(S_0,V_0,e_0,a_0,i_0,R_0\right)\in{\left(\mathbb{R}_0^+\right)}^2\times{\left(L^1{\left(\mathbb{R}_0^+;\mathbb{R}_0^+\right)}\right)}^3\times\mathbb{R}_0^+,$$ we search for $T>0$ and smooth enough $$\left(S,V,e,a,i,R\right)\colon\,\left[0,T\right)\to {\left(\mathbb{R}_0^+\right)}^2\times{\left(L^1{\left(\mathbb{R}_0^+;\mathbb{R}_0^+\right)}\right)}^3\times\mathbb{R}_0^+,$$ such that 
\begin{subequations}
\label{SVEIAR-age}
\begin{align}
&\begin{cases}
\dfrac{\de S}{\de t}=\mu N-\left(p+\int\limits_0^\infty{\beta_A{\left(\theta\right)}a{\left(\,\cdot\,,\theta\right)}+\beta_I{\left(\theta\right)}i{\left(\,\cdot\,,\theta\right)}\,\de\theta}+\mu\right)S\\
S{\left(0\right)}=S_0,\label{SVEIAR-age;a}
\end{cases}\\
&\begin{cases}
\dfrac{\de V}{\de t}=pS-\left(\zeta\epsilon+\int\limits_0^\infty{\beta_A{\left(\theta\right)}a{\left(\,\cdot\,,\theta\right)}+\beta_I{\left(\theta\right)}i{\left(\,\cdot\,,\theta\right)}\,\de\theta}\left(1-\epsilon\right)+\mu\right)V\\
V{\left(0\right)}=V_0,\label{SVEIAR-age;b}
\end{cases}\\
&\begin{cases}
\dfrac{\partial e}{\partial t}+\dfrac{1}{\omega}\dfrac{\partial e}{\partial \theta}=-\left(k+\mu\right)e\\
e{\left(\,\cdot\,,0\right)}=\omega\int\limits_0^\infty{\beta_A{\left(\theta\right)}a{\left(\,\cdot\,,\theta\right)}+\beta_I{\left(\theta\right)}i{\left(\,\cdot\,,\theta\right)}\,\de\theta}\left(S+\left(1-\epsilon\right)V\right)\\
e{\left(0,\,\cdot\,\right)}=e_0,\label{SVEIAR-age;c}
\end{cases}\\
&\begin{cases}
\dfrac{\partial a}{\partial t}+\dfrac{1}{\omega}\dfrac{\partial a}{\partial \theta}=-\left(\gamma_A\xi+\chi\left(1-\xi\right)+\mu\right)a\\
a{\left(\,\cdot\,,0\right)}=\omega\int\limits_0^\infty{k{\left(\theta\right)}q{\left(\theta\right)}e{\left(\,\cdot\,,\theta\right)}\,\de\theta}\\
a{\left(0,\,\cdot\,\right)}=a_0,\label{SVEIAR-age;d}
\end{cases}\\
&\begin{cases}
\dfrac{\partial i}{\partial t}+\dfrac{1}{\omega}\dfrac{\partial i}{\partial \theta}=-\left(\gamma_I+\mu\right)i\\
i{\left(\,\cdot\,,0\right)}=\omega\int\limits_0^\infty{k{\left(\theta\right)}\left(1-q{\left(\theta\right)}\right)e{\left(\,\cdot\,,\theta\right)}+\chi{\left(\theta\right)}\left(1-\xi{\left(\theta\right)}\right)a{\left(\,\cdot\,,\theta\right)}\,\de\theta}\\
i{\left(0,\,\cdot\,\right)}=i_0,\label{SVEIAR-age;e}
\end{cases}\\
&\begin{cases}
\dfrac{\de R}{\de t}=\zeta\epsilon V+\int\limits_0^\infty{\gamma_A{\left(\theta\right)}\xi{\left(\theta\right)}a{\left(\,\cdot\,,\theta\right)}+\gamma_I{\left(\theta\right)}i{\left(\,\cdot\,,\theta\right)}\,\de\theta}-\mu R\\
R{\left(0\right)}=R_0.\label{SVEIAR-age;f}
\end{cases}
\end{align}
\end{subequations}

%
%
%
%
%

The dimensional units of all variables and parameters appeared in $\mathscr{P}$ \eqref{SVEIAR-age} are gathered in \hyperref[table-1]{Table \ref*{table-1}}.

\begin{table}[!h]
\centering 
\begin{tabular}{p{2cm}p{11cm}p{2cm}}
\hline
Independent variables  & Description  & Units  \\ [0.5ex]
\hline         
$t$ & Time   & T \\
$\theta$ & Age, i.e. time elapsed since, e.g., birth or infection & $\Theta$ \\  [1ex]      
\hline
\hline
Conversion factor  & Description  & Units  \\ [0.5ex]
\hline
$\omega$ & Conversion factor from the units of $\theta$ to the units of $t$ & T$\,\Theta^{-1}$\\  [1ex] 
\hline
\hline
Dependent variables  & Description  & Units  \\ [0.5ex]
\hline        
$N$ & Number of total population of individuals  & \#  \\
$S$ & Number of susceptible individuals   & \#  \\
$V$ & Number of vaccinated-with-a-prophylactic-vaccine individuals   & \#  \\
$e$ & Age-density of latent/exposed individuals  &  \#$\,\Theta^{-1}$ \\
$E$ & Number of latent/exposed individuals   & \#  \\
$a$ & Age-density of asymptomatic infectious individuals   &  \#$\,\Theta^{-1}$  \\
$A$ & Number of asymptomatic infectious individuals   & \#  \\
$i$ & Age-density of symptomatic infectious   &   \#$\,\Theta^{-1}$ \\
$I$ & Number of symptomatic infectious individuals   & \#  \\
$R$ & Number of recovered/removed individuals   & \#  \\  [1ex]      
\hline
\hline
Parameters & Description  & Units  \\ [0.5ex]
\hline
$N_0$ & Population size & \# \\
$\mu$ & Birth/Death rate   & T$^{-1}$ \\
$\beta_A$ & Transmission rate of asymptomatic infectious individuals & \#$^{-1}\,$T$^{-1}$ \\
$\beta_I$ & Transmission rate of symptomatic infectious individuals & \#$^{-1}\,$T$^{-1}$ \\
$p$ & Vaccination rate &  T$^{-1}$ \\
$\epsilon$ & Vaccine effectiveness  & -  \\
$\zeta$ & Vaccine-induced immunity rate & T$^{-1}$  \\
$k$ & Latent rate (rate of susceptible individuals becoming infectious)  & T$^{-1}$ \\
$q$ & Proportion of the latent/exposed individuals becoming asymptomatic infectious  & - \\
$\xi$ &  Proportion of the asymptomatic infectious individuals becoming recovered/removed (without developing any symptoms) & - \\
$\chi$ &  Incubation rate (rate of a part of asymptomatic infectious individuals developing symptoms) & T$^{-1}$ \\
$\gamma_A$ & Recovery rate of asymptomatic infectious individuals  &  T$^{-1}$\\
$\gamma_I$ & Recovery rate of symptomatic infectious individuals  &  T$^{-1}$ \\ [1ex]      
\hline
\end{tabular}
\caption{Description of the independent and dependent variables as well as parameters of $\mathscr{M}$, along with their units.}
\label{table-1}
\end{table}

We notice that by integration (with respect to $\theta$ over $\mathbb{R}_0^+$) and summation of the left and right-hand side of the derived ordinary differential equations, one gets 
\begin{equation}
\label{Nconst}
\frac{\de N}{\de t}=0\Leftrightarrow N=N_0\coloneqq S_0+V_0+E_0+A_0+I_0+R_0,
\end{equation}
where $$E_0\coloneqq\int\limits_0^\infty{e_0{\left(\theta\right)}\,\de\theta},\text{ }A_0\coloneqq\int\limits_0^\infty{a_0{\left(\theta\right)}\,\de\theta}\text{ and }I_0\coloneqq\int\limits_0^\infty{i_0{\left(\theta\right)}\,\de\theta}.$$ Hence, an additional hypothesis made is as follows. 
\begin{enumerate}
\setcounter{enumi}{9}
\item The total population remains constant. This is a practical (yet not necessary) assumption, and makes sense when the time-span of the modeled epidemiological phenomenon is way shorter than the time needed for observable changes of the total population (whether they are caused by the epidemic or not).
\end{enumerate} 

Equations \eqref{SVEIAR-age;a}-\eqref{SVEIAR-age;e} are independent of $R$, hence the problem is reduced to the aforementioned subsystem itself. In fact, with \eqref{Nconst} at hand, $R$ can be easily calculated by $$R=N_0-S-V-E-A-I.$$  

\subsubsection{Scaling of age}
\label{sclng}

In order to simplify the analysis of \eqref{SVEIAR-age;a}-\eqref{SVEIAR-age;e}, we eliminate the factor $\omega$. We do so by the \textit{scaling} of the independent age-variable, $\theta$, and turning it to another time-variable measured in the same units as $t$. 

Hence, while keeping the same notation, we change the variables as follows 
\begin{align*}
\omega\theta&\mapsto\theta,\\
\frac{1}{\omega}f\circ\frac{1}{\omega}\mathrm{id}&\mapsto f,\\
g\circ\frac{1}{\omega}\mathrm{id}&\mapsto g,
\end{align*}
for $\left(f,g\right)\in\,\left\{e{\left(t,\,\cdot\,\right)},a{\left(t,\,\cdot\,\right)},i{\left(t,\,\cdot\,\right)}\,\big|\,t\in\mathbb{R}_0^+\right\}\times\left\{\beta_A,\beta_I,k,q,\gamma_A,\xi,\chi,\gamma_I\right\}$, and \eqref{SVEIAR-age;a}-\eqref{SVEIAR-age;e} then becomes 
\begin{subequations}
\label{SVEIAR-age-scl}
\begin{align}
&\begin{cases}
\dfrac{\de S}{\de t}=\mu N_0-\left(p+\int\limits_0^\infty{\beta_A{\left(\theta\right)}a{\left(\,\cdot\,,\theta\right)}+\beta_I{\left(\theta\right)}i{\left(\,\cdot\,,\theta\right)}\,\de\theta}+\mu\right)S\\
S{\left(0\right)}=S_0,\label{SVEIAR-age-scl;a}
\end{cases}\\
&\begin{cases}
\dfrac{\de V}{\de t}=pS-\left(\zeta\epsilon+\int\limits_0^\infty{\beta_A{\left(\theta\right)}a{\left(\,\cdot\,,\theta\right)}+\beta_I{\left(\theta\right)}i{\left(\,\cdot\,,\theta\right)}\,\de\theta}\left(1-\epsilon\right)+\mu\right)V\\
V{\left(0\right)}=V_0,\label{SVEIAR-age-scl;b}
\end{cases}\\
&\begin{cases}
\dfrac{\partial e}{\partial t}+\dfrac{\partial e}{\partial \theta}=-\left(k+\mu\right)e\\
e{\left(\,\cdot\,,0\right)}=\int\limits_0^\infty{\beta_A{\left(\theta\right)}a{\left(\,\cdot\,,\theta\right)}+\beta_I{\left(\theta\right)}i{\left(\,\cdot\,,\theta\right)}\,\de\theta}\left(S+\left(1-\epsilon\right)V\right)\\
e{\left(0,\,\cdot\,\right)}=e_0,\label{SVEIAR-age-scl;c}
\end{cases}\\
&\begin{cases}
\dfrac{\partial a}{\partial t}+\dfrac{\partial a}{\partial \theta}=-\left(\gamma_A\xi+\chi\left(1-\xi\right)+\mu\right)a\\
a{\left(\,\cdot\,,0\right)}=\int\limits_0^\infty{k{\left(\theta\right)}q{\left(\theta\right)}e{\left(\,\cdot\,,\theta\right)}\,\de\theta}\\
a{\left(0,\,\cdot\,\right)}=a_0,\label{SVEIAR-age-scl;d}
\end{cases}\\
&\begin{cases}
\dfrac{\partial i}{\partial t}+\dfrac{\partial i}{\partial \theta}=-\left(\gamma_I+\mu\right)i\\
i{\left(\,\cdot\,,0\right)}=\int\limits_0^\infty{k{\left(\theta\right)}\left(1-q{\left(\theta\right)}\right)e{\left(\,\cdot\,,\theta\right)}+\chi{\left(\theta\right)}\left(1-\xi{\left(\theta\right)}\right)a{\left(\,\cdot\,,\theta\right)}\,\de\theta}\\
i{\left(0,\,\cdot\,\right)}=i_0,\label{SVEIAR-age-scl;e}
\end{cases}
\end{align}
\end{subequations}
where $t$ and (the new) $\theta$ are now measured in the same time-units. 

The flow diagram of the differential equations in \eqref{SVEIAR-age} is shown in \hyperref[fig-1]{Figure \ref*{fig-1}}. 

\begin{figure}[H]
\centering
\includegraphics[width=1\textwidth]{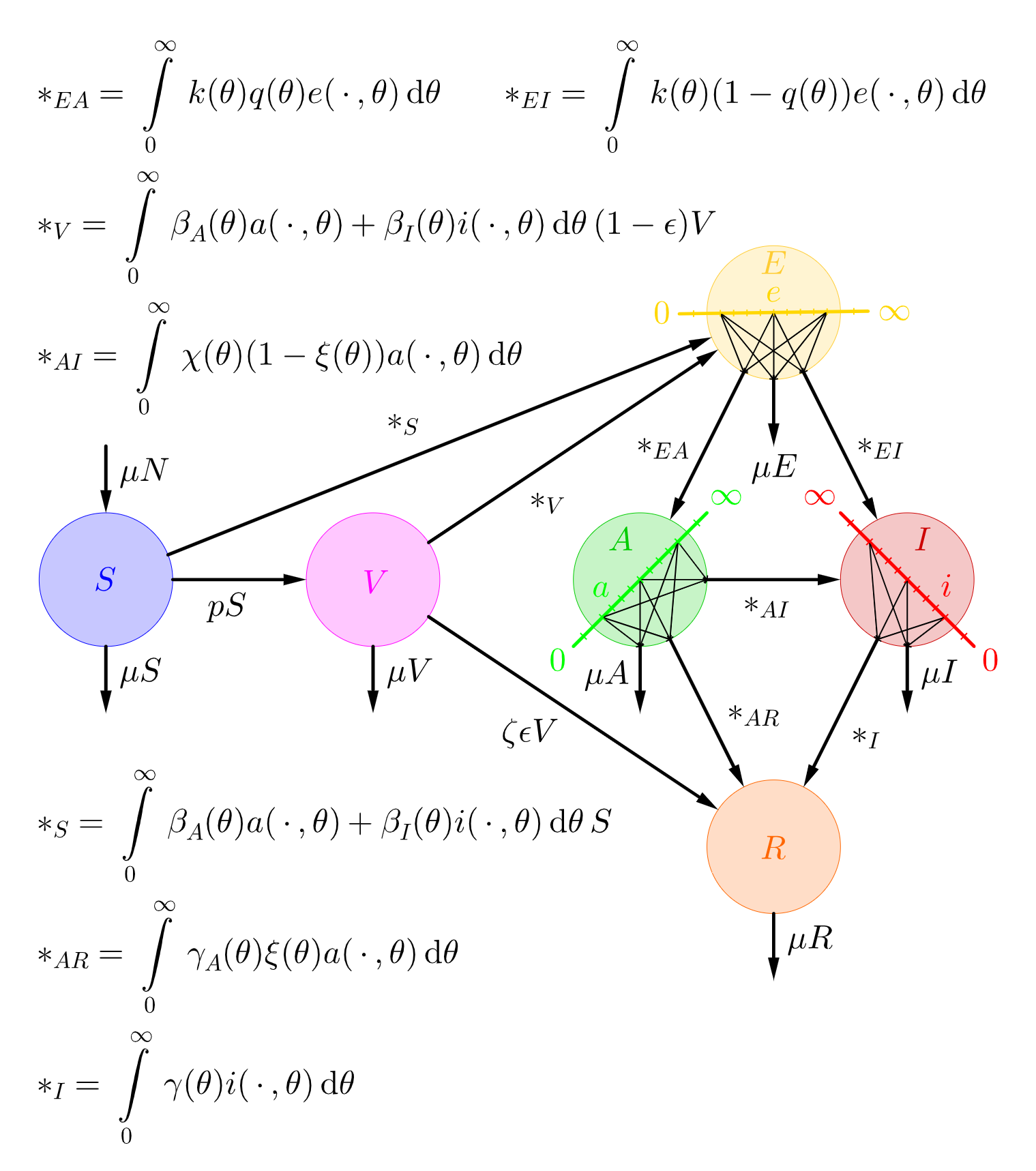}
\caption{Flow diagram of $\mathscr{P}$ \eqref{SVEIAR-age-scl}. The increase/decrease of $A$ and $I$ reflects the outbreak/attenuation of the epidemic, while $E$ is the only source of the aforementioned compartments.}
\label{fig-1}
\end{figure} 

\subsubsection{Global well-posedness}
\label{wllpsdnss}

We set 
\begin{subequations}
\label{grk}
\begin{align}
\beta&\coloneqq \int\limits_0^\infty{\beta_A{\left(\theta\right)}a{\left(\,\cdot\,,\theta\right)}+\beta_I{\left(\theta\right)}i{\left(\,\cdot\,,\theta\right)}\,\de\theta},\label{grk;a}\\
\varepsilon&\coloneqq e{\left(\,\cdot\,,0\right)}=\beta\left(S+\left(1-\epsilon\right)V\right),\label{grk;b}\\
\alpha&\coloneqq a{\left(\,\cdot\,,0\right)}=\int\limits_0^\infty{k{\left(\theta\right)}q{\left(\theta\right)}e{\left(\,\cdot\,,\theta\right)}\,\de\theta},\label{grk;c}\\
\iota&\coloneqq i{\left(\,\cdot\,,0\right)}=\int\limits_0^\infty{k{\left(\theta\right)}\left(1-q{\left(\theta\right)}\right)e{\left(\,\cdot\,,\theta\right)}+\chi{\left(\theta\right)}\left(1-\xi{\left(\theta\right)}\right)a{\left(\,\cdot\,,\theta\right)}\,\de\theta}.\label{grk;d}
\end{align}
\end{subequations}

Integrating the independent variables of \eqref{SVEIAR-age-scl;a} and \eqref{SVEIAR-age-scl;b} along $\left[0,T\right)$, as well as the independent variables of \eqref{SVEIAR-age-scl;c}-\eqref{SVEIAR-age-scl;e} along the \textit{characteristic straight-line paths} $$\left\{\left(t,\theta\right)\in\left[0,T\right)\times\mathbb{R}_0^+\,\big|\,t-\theta=c\right\},\text{ }\forall c\in\mathbb{R},$$ we deduce that 
\begin{subequations}
\label{volt1}
\begin{align}
S{\left(t\right)}&=S_0\e^{-\int\limits_0^t{p+\beta{\left(s\right)}+\mu\,\de s}}+\mu N_0\int\limits_0^t{\e^{-\int\limits_s^t{p+\beta{\left(\tau\right)}+\mu\,\de\tau}}\,\de s},\text{ }\forall t\in\left[0,T\right),\label{volt1;a}\\
V{\left(t\right)}&=V_0\e^{-\int\limits_0^t{\zeta\epsilon+\beta{\left(s\right)}\left(1-\epsilon\right)+\mu\,\de s}}+p\int\limits_0^t{S{\left(s\right)}\e^{-\int\limits_s^t{\zeta\epsilon+\beta{\left(\tau\right)}\left(1-\epsilon\right)+\mu\,\de\tau}}\,\de s},\text{ }\forall t\in\left[0,T\right),\label{volt1;b}\\
e{\left(t,\theta\right)}&=\begin{cases}
e_0{\left(\theta-t\right)}\e^{-\int\limits_0^t{k{\left(\theta-t+s\right)}+\mu\,\de s}},&\text{ if }t\in\left[0,\theta\right)\subsetneq\left[0,T\right)\\
\varepsilon{\left(t-\theta\right)}\e^{-\int\limits_0^\theta{k{\left(s\right)}+\mu\,\de s}},&\text{ if }\theta\in\left[0,t\right)\subsetneq\left[0,T\right),
\end{cases}\label{volt1;c}\\
a{\left(t,\theta\right)}&=\begin{cases}
a_0{\left(\theta-t\right)}\e^{-\int\limits_0^t{\gamma_A{\left(\theta-t+s\right)}\xi{\left(\theta-t+s\right)}+\chi{\left(\theta-t+s\right)}\left(1-\xi{\left(\theta-t+s\right)}\right)+\mu\,\de s}},&\text{ if }t\in\left[0,\theta\right)\subsetneq\left[0,T\right)\\
\alpha{\left(t-\theta\right)}\e^{-\int\limits_0^\theta{\gamma_A{\left(s\right)}\xi{\left(s\right)}+\chi{\left(s\right)}\left(1-\xi{\left(s\right)}\right)+\mu\,\de s}},&\text{ if }\theta\in\left[0,t\right)\subsetneq\left[0,T\right).
\end{cases}\label{volt1;d}\\
i{\left(t,\theta\right)}&=\begin{cases}
i_0{\left(\theta-t\right)}\e^{-\int\limits_0^t{\gamma_I{\left(\theta-t+s\right)}+\mu\,\de s}},&\text{ if }t\in\left[0,\theta\right)\subsetneq\left[0,T\right)\\
\iota{\left(t-\theta\right)}\e^{-\int\limits_0^\theta{\gamma_I{\left(s\right)}+\mu\,\de s}},&\text{ if }\theta\in\left[0,t\right)\subsetneq\left[0,T\right),
\end{cases}\label{volt1;e}
\end{align}
\end{subequations} 

We then plug system \eqref{volt1} into system \eqref{grk} to obtain 
\begin{subequations}
\label{grk2}
\begin{align}
\beta{\left(t\right)}&=\int\limits_0^t{\beta_1{\left(t,s\right)}\,\de s}+\int\limits_0^\infty{\beta_2{\left(t,s\right)}\,\de s},\text{ }\forall t\in\left[0,T\right),\label{grk2;a}\\
\varepsilon{\left(t\right)}&=\beta{\left(t\right)}\left(S{\left(t\right)}+\left(1-\epsilon\right)V{\left(t\right)}\right),\text{ }\forall t\in\left[0,T\right),\label{grk2;b}\\
\alpha{\left(t\right)}&=\int\limits_0^t{\alpha_1{\left(t,s\right)}\,\de s}+\int\limits_0^\infty{\alpha_2{\left(t,s\right)}\,\de s},\text{ }\forall t\in\left[0,T\right),\label{grk2;c}\\
\iota{\left(t\right)}&=\int\limits_0^t{\iota_1{\left(t,s\right)}\,\de s}+\int\limits_0^\infty{\iota_2{\left(t,s\right)}\,\de s},\text{ }\forall t\in\left[0,T\right).\label{grk2;d}
\end{align}
\end{subequations}
where 
\begin{align*}
\beta_1{\left(t,s\right)}&\coloneqq\beta_A{\left(t-s\right)}\alpha{\left(s\right)}\e^{-\int\limits_0^{t-s}{\gamma_A{\left(\tau\right)}\xi{\left(\tau\right)}+\chi{\left(\tau\right)}\left(1-\xi{\left(\tau\right)}\right)+\mu\,\de \tau}}+\\
&\phantom{\coloneqq}+\beta_I{\left(t-s\right)}\iota{\left(s\right)}\e^{-\int\limits_0^{t-s}{\gamma_I{\left(\tau\right)}+\mu\,\de\tau}},\\
\beta_2{\left(t,s\right)}&\coloneqq\beta_A{\left(t+s\right)}a_0{\left(s\right)}\e^{-\int\limits_0^t{\gamma_A{\left(\tau+s\right)}\xi{\left(\tau+s\right)}+\chi{\left(\tau+s\right)}\left(1-\xi{\left(\tau+s\right)}\right)+\mu\,\de\tau}}+\\
&\phantom{\coloneqq}+\beta_I{\left(t+s\right)}i_0{\left(s\right)}\e^{-\int\limits_0^t{\gamma_I{\left(\tau+s\right)}+\mu\,\de\tau}},
\end{align*} 
$S$ and $V$ have already been calculated in terms of $\beta$ in \eqref{volt1;a} and \eqref{volt1;b}, respectively, 
\begin{align*}
\alpha_1{\left(t,s\right)}&\coloneqq k{\left(t-s\right)}q{\left(t-s\right)}\beta{\left(s\right)}\left(S{\left(s\right)}+\left(1-\epsilon\right)V{\left(s\right)}\right)\e^{-\int\limits_0^{t-s}{k{\left(\tau\right)}+\mu\,\de\tau}},\\
\alpha_2{\left(t,s\right)}&\coloneqq k{\left(t+s\right)}q{\left(t+s\right)}e_0{\left(s\right)}\e^{-\int\limits_0^t{k{\left(\tau+s\right)}+\mu\,\de\tau}}
\end{align*} 
and
\begin{align*}
\iota_1{\left(t,s\right)}&\coloneqq k{\left(t-s\right)}\left(1-q{\left(t-s\right)}\right)\beta{\left(s\right)}\left(S{\left(s\right)}+\left(1-\varepsilon\right)V{\left(s\right)}\right)\e^{-\int\limits_0^{t-s}{k{\left(\tau\right)}+\mu\,\de\tau}}+\\
&\phantom{\coloneqq}+\chi{\left(t-s\right)}\left(1-\xi{\left(t-s\right)}\right)\alpha{\left(s\right)}\e^{-\int\limits_0^{t-s}{\gamma_A{\left(\tau\right)}\xi{\left(\tau\right)}+\chi{\left(\tau\right)}\left(1-\xi{\left(\tau\right)}\right)+\mu\,\de\tau}},\\
\iota_2{\left(t,s\right)}&\coloneqq k{\left(t+s\right)}\left(1-q{\left(t+s\right)}\right)e_0{\left(s\right)}\e^{-\int\limits_0^t{k{\left(\tau+s\right)}+\mu\,\de\tau}}+\\
&\phantom{\coloneqq}+\chi{\left(t+s\right)}\left(1-\xi{\left(t+s\right)}\right)a_0{\left(s\right)}\e^{-\int\limits_0^t{\gamma_A{\left(t+s\right)}\xi{\left(t+s\right)}+\chi{\left(t+s\right)}\left(1-\xi{\left(t+s\right)}\right)+\mu\,\de\tau}}.
\end{align*} 

Equations \eqref{grk2;a}, \eqref{grk2;c} and \eqref{grk2;d} can be considered as an auxiliary problem (in integral form) for the unknown functions $\beta$, $\alpha$ and $\iota$. A direct application of the \textit{classic theory of integral equations} provides us with the following preliminary result, the standard proof of which is omitted (see, e.g., \cite{iannelli2017basic,inaba2017age,webb1985theory}).  
\begin{theorem}
\label{glwlp1}
For every $\left(S_0,V_0,e_0,a_0,i_0,R_0\right)\in\mathbb{R}^2\times{\left(L^1{\left(\mathbb{R}_0^+;\mathbb{R}\right)}\right)}^3\times\mathbb{R}$, the problem \eqref{grk2;a}, \eqref{grk2;c}, \eqref{grk2;d} is globally (i.e. $T=\infty$) well-posed, with $\left(\beta,\alpha,\iota\right)\in {\left(C{\left(\mathbb{R}_0^+;\mathbb{R}\right)}\right)}^3$. 
\end{theorem} 

Moreover, it is straightforward to check that if $\left(S_0,V_0,e_0,a_0,i_0,R_0\right)=\left(0,0,0,0,0,N_0\right)$, then the solution of \eqref{grk2;a}, \eqref{grk2;c}, \eqref{grk2;d} is the constant $\left(\beta,\alpha,\iota\right)=\left(0,0,0\right)$. Hence, from the uniqueness of solution we derive the next result. 
\begin{proposition}
\label{glwlp2}
If $\left(S_0,V_0,e_0,a_0,i_0,R_0\right)\in{\left(\mathbb{R}_0^+\right)}^2\times{\left(L^1{\left(\mathbb{R}_0^+;\mathbb{R}_0^+\right)}\right)}^3\times\mathbb{R}_0^+$, then $\left(\beta,\alpha,\iota\right)\in {\left(C{\left(\mathbb{R}_0^+;\mathbb{R}_0^+\right)}\right)}^3$. 
\end{proposition}

The global well-posedness of the main problem then follows. 
\begin{corollary}
\label{glwlp3}
For every $\left(S_0,V_0,e_0,a_0,i_0,R_0\right)\in{\left(\mathbb{R}_0^+\right)}^2\times{\left(L^1{\left(\mathbb{R}_0^+;\mathbb{R}_0^+\right)}\right)}^3\times\mathbb{R}_0^+$, the $\mathscr{P}$ \eqref{SVEIAR-age-scl} is globally well-posed, with $\left(S,V,e,a,i\right)\in{\left(C^1{\left(\mathbb{R}_0^+;\mathbb{R}_0^+\right)}\right)}^2\times {\left(C{\left(\mathbb{R}_0^+;L^1{\left(\mathbb{R}_0^+\right)}\right)}\right)}^3$. In particular, the differential equations in \eqref{SVEIAR-age-scl;a} and \eqref{SVEIAR-age-scl;b} are satisfied $\forall t\in\mathbb{R}_0^+$, while the subsystem \eqref{SVEIAR-age-scl;c}-\eqref{SVEIAR-age-scl;e} is satisfied in the following sense 
\begin{align*}
&\begin{cases}
\lim\limits_{h\to 0}{\dfrac{e{\left(t+h,\theta+h\right)}-e{\left(t,\theta\right)}}{h}}=-\left(k{\left(\theta\right)}+\mu\right)e{\left(t,\theta\right)},\text{ for a.e. }\left(t,\theta\right)\in{\left(\mathbb{R}_0^+\right)}^2\\
e{\left(t,0\right)}=\varepsilon{\left(t\right)},\text{ }\forall t\in\mathbb{R}^+\\
e{\left(0,\theta\right)}=e_0{\left(\theta\right)},\text{ for a.e. }\theta\in\mathbb{R}_0^+,
\end{cases}\\
&\begin{cases}
\lim\limits_{h\to 0}{\dfrac{a{\left(t+h,\theta+h\right)}-a{\left(t,\theta\right)}}{h}}=-\left(\gamma_A{\left(\theta\right)}\xi{\left(\theta\right)}+\chi{\left(\theta\right)}\left(1-\xi{\left(\theta\right)}\right)+\mu\right)a{\left(t,\theta\right)},\text{ for a.e. }\left(t,\theta\right)\in{\left(\mathbb{R}_0^+\right)}^2\\
a{\left(t,0\right)}=\alpha{\left(t\right)},\text{ }\forall t\in\mathbb{R}^+\\
a{\left(0,\theta\right)}=a_0{\left(\theta\right)},\text{ for a.e. }\theta\in\mathbb{R}_0^+,
\end{cases}\\
&\begin{cases}
\lim\limits_{h\to 0}{\dfrac{i{\left(t+h,\theta+h\right)}-i{\left(t,\theta\right)}}{h}}=-\left(\gamma_I{\left(\theta\right)}+\mu\right)i{\left(t,\theta\right)},\text{ for a.e. }\left(t,\theta\right)\in{\left(\mathbb{R}_0^+\right)}^2\\
i{\left(t,0\right)}=\iota{\left(t\right)},\text{ }\forall t\in \mathbb{R}^+\\
i{\left(0,\theta\right)}=i_0{\left(\theta\right)},\text{ for a.e. }\theta\in\mathbb{R}_0^+. 
\end{cases}
\end{align*}
\end{corollary}

We also note that we can obtain certain \textit{regularity results} by strengthening the assumptions regarding the data of the problem, but this lies beyond the scope of the present work. 

\subsubsection{Steady states and basic reproductive number}
\label{eqbrn} 

A \textit{steady state}, $\left(S^*,V^*,e^*,a^*,i^*\right)$, of $\mathscr{P}$ \eqref{SVEIAR-age-scl} is a constant-with-respect-to-$t$ solution, i.e. it is defined to satisfy 
\begin{subequations}
\label{stst0ode}
\begin{align}
&0=\mu N_0-\left(p+\beta^*+\mu\right)S^*,\label{stst0ode;a}\\
&0=p S^*-\left(\zeta\epsilon+\beta^*\left(1-\epsilon\right)+\mu\right)V^*,\label{stst0ode;b}\\
&\begin{cases}
\dfrac{\de e^*}{\de \theta}=-\left(k+\mu\right)e^*\\
e^*{\left(0\right)}=\varepsilon^*,\label{stst0ode;c}
\end{cases}\\
&\begin{cases}
\dfrac{\de a^*}{\de \theta}=-\left(\gamma_A\xi+\chi\left(1-\xi\right)+\mu\right)a^*\\
a^*{\left(0\right)}=\alpha^*,\label{stst0ode;d}
\end{cases}\\
&\begin{cases}
\dfrac{\de i^*}{\de \theta}=-\left(\gamma_I+\mu\right)i^*\\
i^*{\left(0\right)}=\iota^*,\label{stst0ode;e}
\end{cases}
\end{align}
\end{subequations}
where 
\begin{subequations}
\label{stst0}
\begin{align}
\beta^*&\coloneqq\int\limits_0^\infty{\beta_A{\left(\theta\right)}a^*{\left(\theta\right)}+\beta_I{\left(\theta\right)}i^*{\left(\theta\right)}\,\de\theta},\label{stst0;a}\\
\varepsilon^*&\coloneqq\beta^*\left(S^*+\left(1-\epsilon\right)V^*\right),\label{stst0;b}\\
\alpha^*&\coloneqq\int\limits_0^\infty{k{\left(\theta\right)}q{\left(\theta\right)}e^*{\left(\theta\right)}\,\de\theta},\label{stst0;c}\\ 
\iota^*&\coloneqq\int\limits_0^\infty{k{\left(\theta\right)}\left(1-q{\left(\theta\right)}\right)e^*{\left(\theta\right)}+\chi{\left(\theta\right)}\left(1-\xi{\left(\theta\right)}\right)a^*{\left(\theta\right)}\,\de\theta},\label{stst0;d}
\end{align}
\end{subequations}
hence 
\begin{subequations}
\label{stst00}
\begin{align}
S^*&=\frac{\mu N_0}{p+\beta^*+\mu},\label{stst00;a}\\
V^*&=\frac{p S^*}{\zeta\epsilon+\beta^*\left(1-\epsilon\right)+\mu},\label{stst00;b}\\
e^*{\left(\theta\right)}&=\varepsilon^*\e^{-\int\limits_0^\theta{k{\left(s\right)}+\mu\,\de s}},\text{ }\forall\theta\in\mathbb{R}_0^+,\label{stst00;c}\\
a^*{\left(\theta\right)}&=\alpha^*\e^{-\int\limits_0^\theta{\gamma_A{\left(s\right)}\xi{\left(s\right)}+\chi{\left(s\right)}\left(1-\xi{\left(s\right)}\right)+\mu\,\de s}},\text{ }\forall\theta\in\mathbb{R}_0^+,\label{stst00;d}\\
i^*{\left(\theta\right)}&=\iota^*\e^{-\int\limits_0^\theta{\gamma_I{\left(s\right)}+\mu\,\de s}},\text{ }\forall\theta\in\mathbb{R}_0^+.\label{stst00;e}
\end{align}
\end{subequations}
By plugging \eqref{stst0;b}-\eqref{stst0;d} into \eqref{stst00} and expressing the components of a steady state exclusively in terms of the constant parameter $\beta^*$, we get 
\begin{subequations}
\label{stst1}
\begin{align}
S^*&=\frac{\mu N_0}{p+\beta^*+\mu},\label{ststs1;a}\\
V^*&=\frac{p\mu N_0}{\left(p+\beta^*+\mu\right)\left(\zeta\epsilon+\beta^*\left(1-\epsilon\right)+\mu\right)},\label{stst1;b}\\
e^*{\left(\theta\right)}&=\beta^*\frac{\mu N_0}{p+\beta^*+\mu}\left(1+\frac{p\left(1-\epsilon\right)}{\zeta\epsilon+\beta^*\left(1-\epsilon\right)+\mu}\right)\e^{-\int\limits_0^\theta{k{\left(s\right)}+\mu\,\de s}},\text{ }\forall\theta\in\mathbb{R}_0^+,\label{stst1;c}\\
a^*{\left(\theta\right)}&=\beta^*\frac{\mu N_0}{p+\beta^*+\mu}\left(1+\frac{p\left(1-\epsilon\right)}{\zeta\epsilon+\beta^*\left(1-\epsilon\right)+\mu}\right)\int\limits_0^\infty{k{\left(s\right)}q{\left(s\right)}\e^{-\int\limits_0^s{k{\left(\tau\right)}+\mu\,\de\tau}}\,\de s}\times\nonumber\\
&\phantom{=}\times\e^{-\int\limits_0^\theta{\gamma_A{\left(s\right)}\xi{\left(s\right)}+\chi{\left(s\right)}\left(1-\xi{\left(s\right)}\right)+\mu\,\de s}},\text{ }\forall\theta\in\mathbb{R}_0^+,\label{stst1;d}\\
i^*{\left(\theta\right)}&=\beta^*\frac{\mu N_0}{p+\beta^*+\mu}\left(1+\frac{p\left(1-\epsilon\right)}{\zeta\epsilon+\beta^*\left(1-\epsilon\right)+\mu}\right)\times\nonumber\\
&\phantom{=}\times\left(
\begin{aligned}
&\phantom{+}\int\limits_0^\infty{k{\left(s\right)}\left(1-q{\left(s\right)}\right)\e^{-\int\limits_0^s{k{\left(\tau\right)}+\mu\,\de\tau}}\,\de s}+\\
&+\int\limits_0^\infty{k{\left(s\right)}q{\left(s\right)}\e^{-\int\limits_0^s{k{\left(\tau\right)+\mu}\,\de\tau}}\,\de s}\int\limits_0^\infty{\chi{\left(s\right)}\left(1-\xi{\left(s\right)}\right)\e^{-\int\limits_0^s{\gamma_A{\left(\tau\right)}\xi{\left(\tau\right)}+\chi{\left(\tau\right)}\left(1-\xi{\left(\tau\right)}\right)+\mu\,\de\tau}}\,\de s}
\end{aligned}
\right)\times\nonumber\\
&\phantom{=}\times\e^{-\int\limits_0^\theta{\gamma_I{\left(s\right)}+\mu\,\de s}},\text{ }\forall\theta\in\mathbb{R}_0^+.\label{stst1;e}
\end{align}
\end{subequations}

We now set 
\begin{equation}
    \boxed{\mathbb{R}_0^+\ni\mathcal{R}_0\coloneqq\frac{\mu N_0}{p+\mu}\left(1+\frac{p\left(1-\epsilon\right)}{\zeta\epsilon+\mu}\right)\left(\mathcal{R}_A+\mathcal{R}_I\right),}
    \label{R0-definition}
\end{equation}
where 
\begin{equation*}
\boxed{\mathbb{R}_0^+\ni\mathcal{R}_A\coloneqq\int\limits_0^\infty{k{\left(s\right)}q{\left(s\right)}\e^{-\int\limits_0^s{k{\left(\tau\right)}+\mu\,\de\tau}}\,\de s}\int\limits_0^\infty{\beta_A{\left(s\right)}\e^{-\int\limits_0^s{\gamma_A{\left(\tau\right)}\xi{\left(\tau\right)}+\chi{\left(\tau\right)}\left(1-\xi{\left(\tau\right)}\right)+\mu\,\de\tau}}\,\de s}}
\end{equation*}
and 
\begin{equation*}
\boxed{\begin{aligned}
\mathbb{R}_0^+\ni\mathcal{R}_I&\coloneqq\left(
\begin{aligned}
&\phantom{+}\int\limits_0^\infty{k{\left(s\right)}\left(1-q{\left(s\right)}\right)\e^{-\int\limits_0^s{k{\left(\tau\right)}+\mu\,\de\tau}}\,\de s}+\\
&+\int\limits_0^\infty{k{\left(s\right)}q{\left(s\right)}\e^{-\int\limits_0^s{k{\left(\tau\right)+\mu}\,\de\tau}}\,\de s}\int\limits_0^\infty{\chi{\left(s\right)}\left(1-\xi{\left(s\right)}\right)\e^{-\int\limits_0^s{\gamma_A{\left(\tau\right)}\xi{\left(\tau\right)}+\chi{\left(\tau\right)}\left(1-\xi{\left(\tau\right)}\right)+\mu\,\de\tau}}\,\de s}
\end{aligned}
\right)\times\\
&\phantom{\coloneqq}\times\int\limits_0^\infty{\beta_I{\left(s\right)}\e^{-\int\limits_0^s{\gamma_I{\left(\tau\right)}+\mu\,\de\tau}}\,\de s},
\end{aligned}}
\end{equation*}
for the \textit{basic reproductive number} of the aforementioned problem. Its definition emerges naturally from the following result. 
\begin{proposition}
\label{agedthm1} 
Concerning $\beta^*\in\mathbb{R}_0^+$, 
\begin{enumerate}
\item if $\mathcal{R}_0\leq 1$, then $\beta^*=0$, 
\item if $\mathcal{R}_0> 1$, then 
\begin{enumerate}[label=\roman*.]
\item either $\beta^*=0$, 
\item or $\beta^*>0$, such that $$b_2{\beta^*}^2+b_1\beta^*+b_0=0,$$ where 
\begin{align*}
b_2&=\left(1-\epsilon\right),\\
b_1&=\left(p+\mu\right)\left(1-\epsilon\right)+\zeta\epsilon+\mu-\mu N_0\left(1-\epsilon\right)\left(\mathcal{R}_A+\mathcal{R}_I\right),\\
b_0&=\left(p+\mu\right)\left(\epsilon\zeta+\mu\right)\left(1-\mathcal{R}_0\right).
\end{align*} 
\end{enumerate}
\end{enumerate}
\end{proposition} 
\begin{proof}
We substitute $a^*$ and $i^*$ of \eqref{stst1;d} and \eqref{stst1;e}, respectively, into \eqref{stst0;a} to deduce that $$\beta^*=\beta^*\frac{\mu N_0}{p+\beta^*+\mu}\left(1+\frac{p\left(1-\epsilon\right)}{\zeta\epsilon+\beta^*\left(1-\epsilon\right)+\mu}\right)\left(\mathcal{R}_A+\mathcal{R}_I\right).$$ There are only two discrete cases, either $\beta^*=0$, or $\beta^*>0$. If $\beta^*>0$, then, equivalently, $$1=\frac{\mu N_0}{p+\beta^*+\mu}\left(1+\frac{p\left(1-\epsilon\right)}{\zeta\epsilon+\beta^*+\left(1-\epsilon\right)+\mu}\right)\left(\mathcal{R}_A+\mathcal{R}_I\right),$$ or else $$b_2{\beta^*}^2+b_1\beta^*+b_0=0.$$ We observe that 
\begin{enumerate}[label=\alph*.]
\item if $\epsilon=1$ then $b_2=0$ and $b_1>0$,
\item if $\epsilon\neq 1$ then $b_2>0$.
\end{enumerate} 
Therefore, in any case, there exists $\beta^*>0$ satisfying the above equation iff $b_3<0$, i.e. $\mathcal{R}_0>1$, and of course, such $\beta^*$ is unique. 
\end{proof}

The following result is now straightforward. 
\begin{corollary}
\label{agedthm2}
Concerning $\left(S^*,V^*,e^*,a^*,i^*\right)$, 
\begin{enumerate}
\item if $\mathcal{R}_0\leq 1$ then $\left(\e^*,a^*,i^*\right)=\left(0,0,0\right)$, 
\item if $\mathcal{R}_0>1$ then 
\begin{enumerate}[label=\roman*.]
\item either $\left(\e^*,a^*,i^*\right)=\left(0,0,0\right)$,
\item or $\left(\e^*,a^*,i^*\right)>\left(0,0,0\right)$. 
\end{enumerate}
\end{enumerate}
\end{corollary} 
The solution $\left(S^*,V^*,e^*,a^*,i^*\right)$ is called \textit{disease-free steady state} if $\left(\e^*,a^*,i^*\right)=\left(0,0,0\right)$, as well as \textit{endemic steady state} if $\left(\e^*,a^*,i^*\right)>\left(0,0,0\right)$. 

\subsubsection{Global stability}
\label{stblty} 

We are interested in the longer-time dynamics of the modeled epidemiological phenomenon, globally with respect to the set of initial data, ${\left(\mathbb{R}_0^+\right)}^2\times{\left(L^1{\left(\mathbb{R}_0^+;\mathbb{R}_0^+\right)}\right)}^3\times\mathbb{R}_0^+$. Below we check the \textit{global stability} of the steady state of $\mathscr{P}$ \eqref{SVEIAR-age-scl} by finding a \textit{Lyapunov function}. Since the steady state changes with respect to the sign of $1-\mathcal{R}_0$, we check each such case separately.
\begin{theorem}
\label{stglbtm1} 
If $\mathcal{R}_0\leq 1$, then the disease-free steady state, $$\left(S^*,V^*,e^*,a^*,i^*\right)=\left(\frac{\mu N_0}{p+\mu},\frac{p\mu N_0}{\left(p+\mu\right)\left(\zeta\epsilon+\mu\right)},0,0,0\right),$$ is globally asymptotically stable. 
\end{theorem}
\begin{proof}
\textit{Step I:}\\
We define the following functions 
\begin{align*}
f\,\colon{\mathbb{R}^+}&\rightarrow{\mathbb{R}_0^+}\\
x&\mapsto f{\left(x\right)}\coloneqq x-1-\ln{x},
\end{align*}
and
\begin{align*}
L\,\colon{\mathbb{R}_0^+}&\rightarrow{\mathbb{R}_0^+}\\
t&\mapsto L{\left(t;S,V,e,a,i\right)}\coloneqq L_{SV}+L_E+L_A+L_I,
\end{align*}
where 
\begin{align}
&L_{SV}\coloneqq S^*f{\left(\frac{S}{S^*}\right)}+V^*f{\left(\frac{V}{V^*}\right)},\\
&L_E\coloneqq\int\limits_0^\infty{f_E{\left(\theta\right)}e{\left(\,\cdot\,,\theta\right)}\,\de\theta}\\
&L_A\coloneqq \int\limits_0^\infty{f_A{\left(\theta\right)}a{\left(\,\cdot\,,\theta\right)}\,\de\theta}\\
&L_I\coloneqq \int\limits_0^\infty{f_I{\left(\theta\right)}i{\left(\,\cdot\,,\theta\right)}\,\de\theta},
\end{align}
and $f_E$, $f_A$ and $f_I$ are left to be defined.

\textit{Step II:}\\
We differentiate $L_{SV}$, $L_E$, $L_A$ and $L_I$. From \eqref{SVEIAR-age-scl;a} and \eqref{SVEIAR-age-scl;b} we get 
\begin{align*}
\dfrac{\de L_{SV}}{\de t}&=\left(1-\frac{S^*}{S}\right)\dfrac{\de S}{\de t}+\left(1-\frac{V^*}{V}\right)\dfrac{\de V}{\de t}=\\
&=\mu S^*\left(2-\frac{S}{S^*}-\frac{S^*}{S}\right)+p S^*\left(3-\frac{V}{V^*}-\frac{S^*}{S}-\frac{SV^*}{S^*V}\right)-\beta\left(S+\left(1-\epsilon\right)V\right)+\beta\left(S^*+\left(1-\epsilon\right)V^*\right).
\end{align*}
With \eqref{volt1;c} at hand, we also calculate  
\begin{align*}
\dfrac{\de L_E}{\de t}&=\dfrac{\de }{\de t}\left(\int\limits_0^t{f_E{\left(\theta\right)}\varepsilon{\left(t-\theta\right)}\e^{-\int\limits_0^\theta{k{\left(s\right)}+\mu\,\de s}}\,\de\theta}+\int\limits_t^\infty{f_E{\left(\theta\right)}e_0{\left(\theta-t\right)}\e^{-\int\limits_0^t{k{\left(\theta-t+s\right)}+\mu\,\de s}}\,\de\theta}\right)=\\
&=f_E{\left(0\right)}\varepsilon+\int\limits_0^\infty{\left(\dfrac{\de f_E}{\de \theta}{\left(\theta\right)}-f_E{\left(\theta\right)}\left(k{\left(\theta\right)}+\mu\right)\right)e{\left(\,\cdot\,,\theta\right)}\,\de\theta}.
\end{align*}
Similarly, from \eqref{volt1;d} and \eqref{volt1;e} we deduce the following expressions 
\begin{equation*}
\dfrac{\de L_A}{\de t}=f_A{\left(0\right)}\alpha+\int\limits_0^\infty{\left(\dfrac{\de f_A}{\de \theta}{\left(\theta\right)}-f_A{\left(\theta\right)}\left(\gamma_A{\left(\theta\right)}\xi{\left(\theta\right)}+\chi{\left(\theta\right)}\left(1-\xi{\left(\theta\right)}\right)+\mu\right)\right)a{\left(\,\cdot\,,\theta\right)}\,\de\theta}  
\end{equation*}
and
\begin{equation*}
\dfrac{\de L_I}{\de t}=f_I{\left(0\right)}\iota+\int\limits_0^\infty{\left(\dfrac{\de f_I}{\de \theta}{\left(\theta\right)}-f_I{\left(\theta\right)}\left(\gamma_I{\left(\theta\right)}+\mu\right)\right)i{\left(\,\cdot\,,\theta\right)}\,\de\theta}.
\end{equation*}
Therefore, we have
\begin{align*}
\dfrac{\de L}{\de t}&=\dfrac{\de L_{SV}}{\de t}+\dfrac{\de L_E}{\de t}+\dfrac{\de L_A}{\de t}+\dfrac{\de L_I}{\de t}=\\
&=-\mu S^*\left(\frac{S}{S^*}+\frac{S^*}{S}-1\right)-p S^*\left(\frac{V}{V^*}+\frac{S^*}{S}+\frac{SV^*}{S^*V}-3\right)-\left(1-f_E{\left(0\right)}\right)\varepsilon+\\
&\phantom{=}+\left(S^*+\left(1-\epsilon\right)V^*\right)\int\limits_0^\infty{\beta_A{\left(\theta\right)}a{\left(\,\cdot\,,\theta\right)}+\beta_I{\left(\theta\right)}i{\left(\,\cdot\,,\theta\right)}\,\de\theta}+\\
&\phantom{=}+\int\limits_0^\infty{\left(\dfrac{\de f_E}{\de \theta}{\left(\theta\right)}-f_E{\left(\theta\right)}\left(k{\left(\theta\right)}+\mu\right)+f_A{\left(0\right)}k{\left(\theta\right)}q{\left(\theta\right)}+f_I{\left(0\right)}k{\left(\theta\right)}\left(1-q{\left(\theta\right)}\right)\right)e{\left(\,\cdot\,,\theta\right)}\,\de\theta}+\\
&\phantom{=}+\int\limits_0^\infty{\left(\dfrac{\de f_A}{\de \theta}{\left(\theta\right)}-f_A{\left(\theta\right)}\left(\gamma_A{\left(\theta\right)}\xi{\left(\theta\right)}+\chi{\left(\theta\right)}\left(1-\xi{\left(\theta\right)}\right)+\mu\right)+f_I{\left(0\right)}\chi{\left(\theta\right)}\left(1-\xi{\left(\theta\right)}\right)\right)a{\left(\,\cdot\,,\theta\right)}\,\de\theta}+\\
&\phantom{=}+\int\limits_0^\infty{\left(\dfrac{\de f_I}{\de \theta}{\left(\theta\right)}-f_I{\left(\theta\right)}\left(\gamma_I{\left(\theta\right)}+\mu\right)\right)i{\left(\,\cdot\,,\theta\right)}\,\de\theta}.
\end{align*}

\textit{Step III:}\\
We choose $f_E$, $f_A$ and $f_I$ such that the latter terms in the last equation to be zero, that is  
\begin{align*}
\dfrac{\de f_E}{\de \theta}&=f_E\left(k+\mu\right)-f_A{\left(0\right)}kq-f_I{\left(0\right)}k\left(1-q\right),\\
\dfrac{\de f_A}{\de \theta}&=f_A\left(\gamma_A\xi+\chi\left(1-\xi\right)+\mu\right)-f_I{\left(0\right)}\chi\left(1-\xi\right)-\left(S^*+\left(1-\epsilon\right)V^*\right)\beta_A,\\
\dfrac{\de f_I}{\de \theta}&=f_I\left(\gamma_I+\mu\right)-\left(S^*+\left(1-\epsilon\right)V^*\right)\beta_I.
\end{align*}
Hence, $\forall\theta\in\mathbb{R}_0^+$, we set 
\begin{align}
f_E{\left(\theta\right)}&\coloneqq f_A{\left(0\right)}\int\limits_0^\infty{k{\left(s\right)}q{\left(s\right)}\e^{-\int\limits_\theta^s{k{\left(\tau\right)}+\mu\,\de\tau}}\,\de s}+f_I{\left(0\right)}\int\limits_0^\infty{k{\left(s\right)}\left(1-q{\left(s\right)}\right)\e^{-\int\limits_\theta^s{k{\left(\tau\right)}+\mu\,\de\tau}}\,\de s},\label{fe}\\
f_A{\left(\theta\right)}&\coloneqq\left(S^*+\left(1-\epsilon\right)V^*\right)\int\limits_\theta^\infty{\beta_A{\left(s\right)}\e^{-\int\limits_\theta^s{\gamma_A{\left(\tau\right)}\xi{\left(\tau\right)}+\chi{\left(\tau\right)}\left(1-\xi{\left(\tau\right)}\right)+\mu\,\de\tau}}\,\de s}+\nonumber\\
&\phantom{\coloneqq}+f_I{\left(0\right)}\int\limits_\theta^\infty{\chi{\left(s\right)}\left(1-\xi{\left(s\right)}\right)\e^{-\int\limits_\theta^s{\gamma_A{\left(\tau\right)}\xi{\left(\tau\right)}+\chi{\left(\tau\right)}\left(1-\xi{\left(\tau\right)}\right)+\mu\,\de\tau}}\,\de s},\label{fa}\\
f_I{\left(\theta\right)}&\coloneqq\left(S^*+\left(1-\epsilon\right)V^*\right)\int\limits_\theta^\infty{\beta_I{\left(s\right)}\e^{-\int\limits_\theta^s{\gamma_I{\left(\tau\right)}+\mu\,\de\tau}}\,\de s}.\label{fi}
\end{align}
For $f_E$, $f_A$ and $f_I$ defined as such we have
\begin{align*}
\dfrac{\de L}{\de t}&=-\mu S^*\left(\frac{S}{S^*}+\frac{S^*}{S}-1\right)-pS^*\left(\frac{V}{V^*}+\frac{S^*}{S}+\frac{SV^*}{S^*V}-3\right)-\left(1-\mathcal{R}_0\right)\varepsilon.
\end{align*}

\textit{Step IV:}\\
Due to the arithmetic-geometric mean inequality, we derive $$\mathcal{R}_0\leq 1\Rightarrow\,\frac{\de L}{\de t}\leq 0,\text{ }\forall t\in\mathbb{R}_0^+$$ and the equality holds only for the disease-free steady state, i.e. when $$\left(S,V,e,a,i\right)=\left(S^*,V^*,e^*,a^*,i^*\right).$$ Hence, the singleton $\left\{\left(S^*,V^*,e^*,a^*,i^*\right)\right\}$ is the largest invariant set for which $$\frac{\de L}{\de t}=0.$$ Then, from the LaSalle in-variance principle it follows that the disease-free steady state is globally asymptotically stable.
\end{proof}

\begin{theorem}
\label{stglbtm2} 
If $\mathcal{R}_0>1$, then the endemic steady state, $$\left(S^*,V^*,e^*,a^*,i^*\right)\neq\left(S^*,V^*,0,0,0\right),$$ is globally asymptotically stable. 
\end{theorem}
\begin{proof}
\textit{Step I:}\\
Based on \eqref{fe}-\eqref{fi}, we now define 
\begin{align*}
L\,\colon{\mathbb{R}_0^+}&\rightarrow{\mathbb{R}_0^+}\\
t&\mapsto L{\left(t;S,V,e,a,i\right)}\coloneqq L_{SV}+L_E+L_A+L_I,
\end{align*}
with  
\begin{align*}
L_{SV}&\coloneqq S^*f{\left(\frac{S}{S^*}\right)}+V^*f{\left(\frac{V}{V^*}\right)},\\
L_E&\coloneqq f_A{\left(0\right)}\int\limits_0^\infty{\int\limits_\theta^\infty{k{\left(s\right)}q{\left(s\right)}e^*{\left(s\right)}\,\de s}\,f{\left(\frac{e{\left(\,\cdot\,,\theta\right)}}{e^*{\left(\theta\right)}}\right)}\,\de \theta}+\\
&\phantom{\coloneqq}+f_I{\left(0\right)}\int\limits_0^\infty{\int\limits_\theta^\infty{k{\left(s\right)}\left(1-q{\left(s\right)}\right)e^*{\left(s\right)}\,\de s}\,f{\left(\frac{e{\left(\,\cdot\,,\theta\right)}}{e^*{\left(\theta\right)}}\right)}\,\de \theta},\\
L_A&\coloneqq \left(S^*+\left(1-\epsilon\right)V^*\right)\int\limits_0^\infty{\int\limits_\theta^\infty{\beta_A{\left(s\right)}a^*{\left(s\right)}\,\de s}\,f{\left(\frac{a{\left(\,\cdot\,,\theta\right)}}{a^*{\left(\theta\right)}}\right)}\,\de \theta}+\\
&\phantom{\coloneqq}+f_I{\left(0\right)}\int\limits_0^\infty{\int\limits_\theta^\infty{\chi{\left(s\right)}\left(1-\xi{\left(s\right)}\right)a^*{\left(s\right)}\,\de s}\,f{\left(\frac{a{\left(\,\cdot\,,\theta\right)}}{a^*{\left(\theta\right)}}\right)}\,\de \theta},\\
L_I&\coloneqq \left(S^*+\left(1-\epsilon\right)V^*\right)\int\limits_0^\infty{\int\limits_\theta^\infty{\beta_I{\left(s\right)}i^*{\left(s\right)}\,\de s}\,f{\left(\frac{i{\left(\,\cdot\,,\theta\right)}}{i^*{\left(\theta\right)}}\right)}\,\de \theta} 
\end{align*}
and 
\begin{align*}
f\,\colon{\mathbb{R}^+}&\rightarrow{\mathbb{R}_0^+}\\
x&\mapsto f{\left(x\right)}\coloneqq x-1-\ln{x}.
\end{align*}
\textit{Step IIa:}\\
We differentiate $L_{SV}$, $L_E$, $L_A$ and $L_I$. From \eqref{SVEIAR-age-scl;a} and \eqref{SVEIAR-age-scl;b} we get 
\begin{align*}
\dfrac{\de L_{SV}}{\de t}&=\left(1-\frac{S^*}{S}\right)\dfrac{\de S}{\de t}+\left(1-\frac{V^*}{V}\right)\dfrac{\de V}{\de t}=\\
&=-\mu S^*\left(\frac{S}{S^*}+\frac{S^*}{S}-2\right)-p S^*\left(\frac{V}{V^*}+\frac{S^*}{S}+\frac{SV^*}{S^*V}-3\right)+\\
&\phantom{=}+S^*{\int\limits_0^\infty{\beta_A{\left(\theta\right)}a^*{\left(\theta\right)}\left(1-\frac{S\,a{\left(\,\cdot\,,\theta\right)}}{S^*a^*{\left(\theta\right)}}-\frac{S^*}{S}+\frac{a{\left(\,\cdot\,,\theta\right)}}{a^*{\left(\theta\right)}}\right)}\,\de \theta}+\\
&\phantom{=}+S^*{\int\limits_0^\infty{\beta_I{\left(\theta\right)}i^*{\left(\theta\right)}\left(1-\frac{S\,i{\left(\,\cdot\,,\theta\right)}}{S^*i^*{\left(\theta\right)}}-\frac{S^*}{S}+\frac{i{\left(\,\cdot\,,\theta\right)}}{i^*{\left(\theta\right)}}\right)}\,\de \theta}+\\
&\phantom{=}+\left(1-\epsilon\right)V^*{\int\limits_0^\infty{\beta_A{\left(\theta\right)}a^*{\left(\theta\right)}\left(-1-\frac{V\,a{\left(\,\cdot\,,\theta\right)}}{V^*a^*{\left(\theta\right)}}+\frac{V}{V^*}+\frac{a{\left(\,\cdot\,,\theta\right)}}{a^*{\left(\theta\right)}}\right)}\,\de \theta}+\\
&\phantom{=}+\left(1-\epsilon\right)V^*{\int\limits_0^\infty{\beta_I{\left(\theta\right)}i^*{\left(\theta\right)}\left(-1-\frac{V\,i{\left(\,\cdot\,,\theta\right)}}{V^*i^*{\left(\theta\right)}}+\frac{V}{V^*}+\frac{i{\left(\,\cdot\,,\theta\right)}}{i^*{\left(\theta\right)}}\right)}\,\de \theta}.
\end{align*}
From the differential equation in \eqref{SVEIAR-age-scl;c} along with \eqref{stst0ode;c} we have
\begin{equation*}
\dfrac{\partial }{\partial t}f{\left(\frac{e{\left(t,\theta\right)}}{e^*{\left(\theta\right)}}\right)}=-\dfrac{\partial }{\partial \theta}f{\left(\frac{e{\left(t,\theta\right)}}{e^*{\left(\theta\right)}}\right)},
\end{equation*}
thus 
\begin{align*}
\dfrac{\de L_{E}}{\de t}&=-f_A{\left(0\right)}\int\limits_0^\infty{\int\limits_\theta^\infty{k{\left(s\right)}q{\left(s\right)}e^*{\left(s\right)}\,\de s}\,\dfrac{\partial }{\partial \theta}f{\left(\frac{e{\left(\,\cdot\,,\theta\right)}}{e^*{\left(\theta\right)}}\right)}\,\de \theta}-\\
&\phantom{=}-f_I{\left(0\right)}\int\limits_0^\infty{\int\limits_\theta^\infty{k{\left(s\right)}\left(1-q{\left(s\right)}\right)e^*{\left(s\right)}\,\de s}\,\dfrac{\partial }{\partial \theta}f{\left(\frac{e{\left(\,\cdot\,,\theta\right)}}{e^*{\left(\theta\right)}}\right)}\,\de \theta}=\\
&=f_A{\left(0\right)}{\int\limits_0^\infty{k{\left(\theta\right)}q{\left(\theta\right)}e^*{\left(\theta\right)}\left(\frac{\varepsilon}{\varepsilon^*}-\frac{e{\left(\,\cdot\,,\theta\right)}}{e^*{\left(\theta\right)}}+\ln{\frac{e{\left(\,\cdot\,,\theta\right)}}{e^*{\left(\theta\right)}}}-\ln{\frac{\varepsilon}{\varepsilon^*}}\right)}\,\de \theta}+\\
&\phantom{=}+f_I{\left(0\right)}{\int\limits_0^\infty{k{\left(\theta\right)}\left(1-q{\left(\theta\right)}\right)e^*{\left(\theta\right)}\left(\frac{\varepsilon}{\varepsilon^*}-\frac{e{\left(\,\cdot\,,\theta\right)}}{e^*{\left(\theta\right)}}+\ln{\frac{e{\left(\,\cdot\,,\theta\right)}}{e^*{\left(\theta\right)}}}-\ln{\frac{\varepsilon}{\varepsilon^*}}\right)}\,\de \theta}.
\end{align*}
Similarly, from \eqref{SVEIAR-age-scl;d} along with \eqref{stst0ode;d}, as well as \eqref{SVEIAR-age-scl;e} along with \eqref{stst0ode;e}, we deduce 
\begin{align*}
\dfrac{\de L_{A}}{\de t}&=
\left(S^*+\left(1-\epsilon\right)V^*\right){\int\limits_0^\infty{\beta_A{\left(\theta\right)}a^*{\left(\theta\right)}\left(\frac{\alpha}{\alpha^*}-\frac{a{\left(\,\cdot\,,\theta\right)}}{a^*{\left(\theta\right)}}+\ln{\frac{a{\left(\,\cdot\,,\theta\right)}}{a^*{\left(\theta\right)}}}-\ln{\frac{\alpha}{\alpha^*}}\right)}\,\de \theta}+\\
&\phantom{=}+f_I{\left(0\right)}{\int\limits_0^\infty{\chi{\left(\theta\right)}\left(1-\xi{\left(\theta\right)}\right)a^*{\left(\theta\right)}\left(\frac{\alpha}{\alpha^*}-\frac{a{\left(\,\cdot\,,\theta\right)}}{a^*{\left(\theta\right)}}+\ln{\frac{a{\left(\,\cdot\,,\theta\right)}}{a^*{\left(\theta\right)}}}-\ln{\frac{\alpha}{\alpha^*}}\right)}\,\de \theta}
\end{align*}
and
\begin{align*}
\dfrac{\de L_{I}}{\de t}&=
\left(S^*+\left(1-\epsilon\right)V^*\right){\int\limits_0^\infty{\beta_I{\left(\theta\right)}i^*{\left(\theta\right)}\left(\frac{\iota}{\iota^*}-\frac{i{\left(\,\cdot\,,\theta\right)}}{i^*{\left(\theta\right)}}+\ln{\frac{i{\left(\,\cdot\,,\theta\right)}}{i^*{\left(\theta\right)}}}-\ln{\frac{\iota}{\iota^*}}\right)}\,\de \theta},
\end{align*}
respectively. Therefore, we have
\begin{align*}
\dfrac{\de L}{\de t}&=\dfrac{\de L_{SV}}{\de t}+\dfrac{\de L_E}{\de t}+\dfrac{\de L_A}{\de t}+\dfrac{\de L_I}{\de t}=\\
&=-\mu S^*\left(\frac{S}{S^*}+\frac{S^*}{S}-2\right)-pS^*\left(\frac{V}{V^*}+\frac{S^*}{S}+\frac{SV^*}{S^*V}-3\right)+\\
&\phantom{=}+S^*{\int\limits_0^\infty{\beta_A{\left(\theta\right)}a^*{\left(\theta\right)}\left(1-\frac{S^*}{S}+\ln{\frac{a{\left(\,\cdot\,,\theta\right)}}{a^*{\left(\theta\right)}}}-\ln{\frac{\alpha}{\alpha^*}}\right)}\,\de \theta}+\\
&\phantom{=}+S^*{\int\limits_0^\infty{\beta_I{\left(\theta\right)}i^*{\left(\theta\right)}\left(1-\frac{S^*}{S}+\ln{\frac{i{\left(\,\cdot\,,\theta\right)}}{i^*{\left(\theta\right)}}}-\ln{\frac{\iota}{\iota^*}}\right)}\,\de \theta}+\\
&\phantom{=}+\left(1-\epsilon\right)V^*{\int\limits_0^\infty{\beta_A{\left(\theta\right)}a^*{\left(\theta\right)}\left(-1+\frac{V}{V^*}+\ln{\frac{a{\left(\,\cdot\,,\theta\right)}}{a^*{\left(\theta\right)}}}-\ln{\frac{\alpha}{\alpha^*}}\right)}\,\de \theta}+\\
&\phantom{=}+\left(1-\epsilon\right)V^*{\int\limits_0^\infty{\beta_I{\left(\theta\right)}i^*{\left(\theta\right)}\left(-1+\frac{V}{V^*}+\ln{\frac{i{\left(\,\cdot\,,\theta\right)}}{i^*{\left(\theta\right)}}}-\ln{\frac{\iota}{\iota^*}}\right)}\,\de \theta}+\\
&\phantom{=}+f_A{\left(0\right)}{\int\limits_0^\infty{k{\left(\theta\right)}q{\left(\theta\right)}e^*{\left(\theta\right)}\left(\ln{\frac{e{\left(\,\cdot\,,\theta\right)}}{e^*{\left(\theta\right)}}}-\ln{\frac{\varepsilon}{\varepsilon^*}}\right)}\,\de \theta}+\\
&\phantom{=}+f_I{\left(0\right)}{\int\limits_0^\infty{k{\left(\theta\right)}\left(1-q{\left(\theta\right)}\right)e^*{\left(\theta\right)}\left(\ln{\frac{e{\left(\,\cdot\,,\theta\right)}}{e^*{\left(\theta\right)}}}-\ln{\frac{\varepsilon}{\varepsilon^*}}\right)}\,\de \theta}+\\
&\phantom{=}+f_I{\left(0\right)}{\int\limits_0^\infty{\chi{\left(\theta\right)}\left(1-\xi{\left(\theta\right)}\right)a^*{\left(\theta\right)}\left(\ln{\frac{a{\left(\,\cdot\,,\theta\right)}}{a^*{\left(\theta\right)}}}-\ln{\frac{\alpha}{\alpha^*}}\right)}\,\de \theta}+\sum\limits_{i=1}^6{D_i},
\end{align*}
where 
\begin{align*}
D_1&\coloneqq\left(S^*+\left(1-\epsilon\right)V^*\right){\int\limits_0^\infty{\beta_A{\left(\theta\right)}a^*{\left(\theta\right)}\frac{\alpha}{\alpha^*}+\beta_I{\left(\theta\right)}i^*{\left(\theta\right)}\frac{\iota}{\iota^*}}\,\de \theta},\\
D_2&\coloneqq-\left(S^*+\left(1-\epsilon\right)V^*\right){\int\limits_0^\infty{\beta_A{\left(\theta\right)}\e^{-\int\limits_0^\theta{\gamma_A{\left(s\right)}\xi{\left(s\right)}+\chi{\left(s\right)}\left(1-\xi{\left(s\right)}\right)+\mu\,\de s}}}\,\de \theta}{\int\limits_0^\infty{k{\left(\theta\right)}q{\left(\theta\right)}e^*{\left(\theta\right)}}\frac{e{\left(\,\cdot\,,\theta\right)}}{e^*{\left(\theta\right)}}\,\de \theta}-\\
&\phantom{\coloneqq}-f_I{\left(0\right)}\left({\int\limits_0^\infty{k{\left(\theta\right)}\left(1-q{\left(\theta\right)}\right)e^*{\left(\theta\right)}{\frac{e{\left(\,\cdot\,,\theta\right)}}{e^*{\left(\theta\right)}}}}\,\de \theta}+{\int\limits_0^\infty{\chi{\left(\theta\right)}\left(1-\xi{\left(\theta\right)}\right)a^*{\left(\theta\right)}{\frac{a{\left(\,\cdot\,,\theta\right)}}{a^*{\left(\theta\right)}}}}\,\de \theta}\right),\\
D_3&\coloneqq-S^*{\int\limits_0^\infty{\beta_A{\left(\theta\right)}a^*{\left(\theta\right)}\frac{S\,a{\left(\,\cdot\,,\theta\right)}}{S^*a^*{\left(\theta\right)}}+\beta_I{\left(\theta\right)}i^*{\left(\theta\right)}\frac{S\,i{\left(\,\cdot\,,\theta\right)}}{S^*i^*{\left(\theta\right)}}}\,\de \theta}-\\
&\phantom{\coloneqq}-\left(1-\epsilon\right)V^*{\int\limits_0^\infty{\beta_A{\left(\theta\right)}a^*{\left(\theta\right)}\frac{V\,a{\left(\,\cdot\,,\theta\right)}}{V^*a^*{\left(\theta\right)}}+\beta_I{\left(\theta\right)}i^*{\left(\theta\right)}\frac{V\,i{\left(\,\cdot\,,\theta\right)}}{V^*i^*{\left(\theta\right)}}}\,\de \theta},\\
D_4&\coloneqq\frac{\varepsilon}{\varepsilon^*}{\int\limits_0^\infty{\left(f_A{\left(0\right)}k{\left(\theta\right)}q{\left(\theta\right)}+f_I{\left(0\right)}k{\left(\theta\right)}\left(1-q{\left(\theta\right)}\right)\right)e^*{\left(\theta\right)}}\,\de \theta},\\
D_5&\coloneqq-f_I{\left(0\right)}{\int\limits_0^\infty{\chi{\left(\theta\right)}\left(1-\xi{\left(\theta\right)}\right)\e^{-\int\limits_0^\theta{\gamma_A{\left(s\right)}\xi{\left(s\right)}+\chi{\left(s\right)}\left(1-\xi{\left(s\right)}\right)+\mu\,\de s}}}\,\de \theta}{\int\limits_0^\infty{k{\left(\theta\right)}q{\left(\theta\right)}e^*{\left(\theta\right)}}\frac{e{\left(\,\cdot\,,\theta\right)}}{e^*{\left(\theta\right)}}\,\de \theta},\\
D_6&\coloneqq f_I{\left(0\right)}{\int\limits_0^\infty{\chi{\left(\theta\right)}\left(1-\xi{\left(\theta\right)}\right)a^*{\left(\theta\right)}{\frac{\alpha}{\alpha^*}}}\,\de \theta}.
\end{align*}
\textit{Step IIb:}\\
From \eqref{grk;c} and \eqref{stst00;d} we see that $$D_5+D_6=0.$$ Moreover, by \eqref{stst0} and \eqref{stst00}, along with \eqref{grk;a} and \eqref{grk;b}, we observe that
\begin{align*}
D_4&=\frac{\varepsilon}{\varepsilon^*}\left(S^*+\left(1-\epsilon\right)V^*\right)\int\limits_0^\infty{\beta_A{\left(\theta\right)}\e^{-\int\limits_0^\theta{\gamma_A{\left(s\right)}\xi{\left(s\right)}+\chi{\left(s\right)}\left(1-\xi{\left(s\right)}\right)+\mu\,\de s}}\,\de\theta}{\int\limits_0^\infty{k{\left(\theta\right)}q{\left(\theta\right)}e^*{\left(\theta\right)}}\,\de \theta}+\\
&\phantom{=}+\frac{\varepsilon}{\varepsilon^*}f_I{\left(0\right)}\left(\int\limits_0^\infty{k{\left(\theta\right)}\left(1-q{\left(\theta\right)}\right)e^*{\left(\theta\right)}+\chi{\left(\theta\right)}\left(1-\xi{\left(\theta\right)}\right)a^*{\left(\theta\right)}\,\de\theta}\right)=\\
&=\frac{\varepsilon}{\varepsilon^*}\left(S^*+\left(1-\epsilon\right)V^*\right)\left(\alpha^*\int\limits_0^\infty{\beta_A{\left(\theta\right)}\e^{-\int\limits_0^\theta{\gamma_A{\left(s\right)}\xi{\left(s\right)}+\chi{\left(s\right)}\left(1-\xi{\left(s\right)}\right)+\mu\,\de s}}\,\de\theta}+\iota^*\int\limits_0^\infty{\beta_I{\left(\theta\right)}\e^{-\int\limits_0^\theta{\gamma_I{\left(s\right)}+\mu\,\de s}}\,\de\theta}\right)=\\
&=\frac{\varepsilon}{\varepsilon^*}\left(S^*+\left(1-\epsilon\right)V^*\right)\beta^*=\frac{\varepsilon}{\varepsilon^*}\varepsilon^*=\beta\left(S+\left(1-\epsilon\right)V\right)=\\
&=\left(S+\left(1-\epsilon\right)V\right)\int\limits_0^\infty{\beta_A{\left(\theta\right)}a{\left(\,\cdot\,,\theta\right)}+\beta_I{\left(\theta\right)}i{\left(\,\cdot\,,\theta\right)}\,\de\theta}=-D_3. 
\end{align*}
Additionally, from \eqref{stst0} and \eqref{stst00} we have that
\begin{align*}
-D_2&=\alpha\left(S^*+\left(1-\epsilon\right)V^*\right){\int\limits_0^\infty{\beta_A{\left(\theta\right)}\e^{-\int\limits_0^\theta{\gamma_A{\left(s\right)}\xi{\left(s\right)}+\chi{\left(s\right)}\left(1-\xi{\left(s\right)}\right)+\mu\,\de s}}}\,\de \theta}+\iota f_I{\left(0\right)}=\\
&=\dfrac{\alpha}{\alpha^*}\alpha^*\left(S^*+\left(1-\epsilon\right)V^*\right){\int\limits_0^\infty{\beta_A{\left(\theta\right)}\e^{-\int\limits_0^\theta{\gamma_A{\left(s\right)}\xi{\left(s\right)}+\chi{\left(s\right)}\left(1-\xi{\left(s\right)}\right)+\mu\,\de s}}}\,\de \theta}+\dfrac{\iota}{\iota^*}\iota^*f_I{\left(0\right)}=\\
&=\dfrac{\alpha}{\alpha^*}\left(S^*+\left(1-\epsilon\right)V^*\right){\int\limits_0^\infty{\beta_A{\left(\theta\right)}\e^{-\int\limits_0^\theta{\gamma_A{\left(s\right)}\xi{\left(s\right)}+\chi{\left(s\right)}\left(1-\xi{\left(s\right)}\right)+\mu\,\de s}}}\,\de \theta}{\int\limits_0^\infty{k{\left(\theta\right)}q{\left(\theta\right)}e^*{\left(\theta\right)}}\,\de \theta}+\\
&\phantom{=}+\dfrac{\iota}{\iota^*}f_I{\left(0\right)}\left({\int\limits_0^\infty{k{\left(\theta\right)}\left(1-q{\left(\theta\right)}\right)e^*{\left(\theta\right)}}\,\de \theta}+{\int\limits_0^\infty{\chi{\left(\theta\right)}\left(1-\xi{\left(\theta\right)}\right)a^*{\left(\theta\right)}}\,\de \theta}\right)=D_1.
\end{align*}
Consequently, $$\sum\limits_{i=1}^6{D_i}=0$$ and with \eqref{stst0ode;b} at hand we deduce that 
\begin{align*}
\dfrac{\de L}{\de t}&=-\mu S^*\left(\frac{S}{S^*}+\frac{S^*}{S}-2\right)-\left(\zeta\epsilon+\beta^*\left(1-\epsilon\right)+\mu\right)V^*\left(\frac{V}{V^*}+\frac{S^*}{S}+\frac{SV^*}{S^*V}-3\right)+\\
&\phantom{=}+S^*{\int\limits_0^\infty{\beta_A{\left(\theta\right)}a^*{\left(\theta\right)}\left(1-\frac{S^*}{S}+\ln{\frac{a{\left(\,\cdot\,,\theta\right)}}{a^*{\left(\theta\right)}}}-\ln{\frac{\alpha}{\alpha^*}}\right)}\,\de \theta}+\\
&\phantom{=}+S^*{\int\limits_0^\infty{\beta_I{\left(\theta\right)}i^*{\left(\theta\right)}\left(1-\frac{S^*}{S}+\ln{\frac{i{\left(\,\cdot\,,\theta\right)}}{i^*{\left(\theta\right)}}}-\ln{\frac{\iota}{\iota^*}}\right)}\,\de \theta}+\\
&\phantom{=}+\left(1-\epsilon\right)V^*{\int\limits_0^\infty{\beta_A{\left(\theta\right)}a^*{\left(\theta\right)}\left(-1+\frac{V}{V^*}+\ln{\frac{a{\left(\,\cdot\,,\theta\right)}}{a^*{\left(\theta\right)}}}-\ln{\frac{\alpha}{\alpha^*}}\right)}\,\de \theta}+\\
&\phantom{=}+\left(1-\epsilon\right)V^*{\int\limits_0^\infty{\beta_I{\left(\theta\right)}i^*{\left(\theta\right)}\left(-1+\frac{V}{V^*}+\ln{\frac{i{\left(\,\cdot\,,\theta\right)}}{i^*{\left(\theta\right)}}}-\ln{\frac{\iota}{\iota^*}}\right)}\,\de \theta}+\\
&\phantom{=}+f_A{\left(0\right)}{\int\limits_0^\infty{k{\left(\theta\right)}q{\left(\theta\right)}e^*{\left(\theta\right)}\left(\ln{\frac{e{\left(\,\cdot\,,\theta\right)}}{e^*{\left(\theta\right)}}}-\ln{\frac{\varepsilon}{\varepsilon^*}}\right)}\,\de \theta}+\\
&\phantom{=}+f_I{\left(0\right)}{\int\limits_0^\infty{k{\left(\theta\right)}\left(1-q{\left(\theta\right)}\right)e^*{\left(\theta\right)}\left(\ln{\frac{e{\left(\,\cdot\,,\theta\right)}}{e^*{\left(\theta\right)}}}-\ln{\frac{\varepsilon}{\varepsilon^*}}\right)}\,\de \theta}+\\
&\phantom{=}+f_I{\left(0\right)}{\int\limits_0^\infty{\chi{\left(\theta\right)}\left(1-\xi{\left(\theta\right)}\right)a^*{\left(\theta\right)}\left(\ln{\frac{a{\left(\,\cdot\,,\theta\right)}}{a^*{\left(\theta\right)}}}-\ln{\frac{\alpha}{\alpha^*}}\right)}\,\de \theta}.
\end{align*}
\textit{Step IIc:}\\ 
We then proceed by adding some useful zero terms in the above equation. First, from \eqref{grk;b}, along with \eqref{stst0;a} and \eqref{stst0;b}, we have the following useful expression for $$\left(S^*+\left(1-\epsilon\right)V^*\right){\int\limits_0^\infty{\beta_A{\left(\theta\right)}a^*{\left(\theta\right)}+\beta_I{\left(\theta\right)}i^*{\left(\theta\right)}}\,\de \theta}=\varepsilon^*$$ as follows 
\begin{align*}
\varepsilon^*&=\dfrac{\varepsilon^*}{\varepsilon}\varepsilon=\dfrac{\varepsilon^*}{\varepsilon}\left(S+\left(1-\epsilon\right)V\right){\int\limits_0^\infty{\beta_A{\left(\theta\right)}a{\left(\,\cdot\,,\theta\right)}+\beta_I{\left(\theta\right)}i{\left(\,\cdot\,,\theta\right)}}\,\de \theta}=\\
&=S^*{\int\limits_0^\infty{\beta_A{\left(\theta\right)}a^*{\left(\theta\right)}\frac{S\,a{\left(\,\cdot\,,\theta\right)}\varepsilon^*}{S^*a^*{\left(\theta\right)}\varepsilon}+\beta_I{\left(\theta\right)}i^*{\left(\theta\right)}\frac{S\,i{\left(\,\cdot\,,\theta\right)}\varepsilon^*}{S^*i^*{\left(\theta\right)}\varepsilon}}\,\de \theta}+\\
&\phantom{=}+\left(1-\epsilon\right)V^*{\int\limits_0^\infty{\beta_A{\left(\theta\right)}a^*{\left(\theta\right)}\frac{V\,a{\left(\,\cdot\,,\theta\right)}\varepsilon^*}{V^*a^*{\left(\theta\right)}\varepsilon}+\beta_I{\left(\theta\right)}i^*{\left(\theta\right)}\frac{V\,i{\left(\,\cdot\,,\theta\right)}\varepsilon^*}{V^*i^*{\left(\theta\right)}\varepsilon}}\,\de \theta}.
\end{align*}
Second, by \eqref{grk;c} and \eqref{grk;d}, along with \eqref{stst0;c} and \eqref{stst0;d}, we have
\begin{align*}
0&=\left(S^*+\left(1-\epsilon\right)V^*\right){\int\limits_0^\infty{\beta_A{\left(\theta\right)}\e^{-\int\limits_0^\theta{\gamma_A{\left(s\right)}\xi{\left(s\right)}+\chi{\left(s\right)}\left(1-\xi{\left(s\right)}\right)+\mu\,\de s}}}\,\de \theta}\left(\alpha^*-\dfrac{\alpha^*}{\alpha}\alpha\right)+f_I{\left(0\right)}\left(\iota^*-\dfrac{\iota^*}{\iota}\iota\right)=\\
&=\left(S^*+\left(1-\epsilon\right)V^*\right){\int\limits_0^\infty{\beta_A{\left(\theta\right)}\e^{-\int\limits_0^\theta{\gamma_A{\left(s\right)}\xi{\left(s\right)}+\chi{\left(s\right)}\left(1-\xi{\left(s\right)}\right)+\mu\,\de s}}}\,\de \theta}{\int\limits_0^\infty{k{\left(\theta\right)}q{\left(\theta\right)}e^*{\left(\theta\right)}}\left(1-\frac{e{\left(\,\cdot\,,\theta\right)}\alpha^*}{e^*{\left(\theta\right)}\alpha}\right)\,\de \theta}+\\
&\phantom{=}+f_I{\left(0\right)}\left({\int\limits_0^\infty{k{\left(\theta\right)}\left(1-q{\left(\theta\right)}\right)e^*{\left(\theta\right)}{\left(1-\frac{e{\left(\,\cdot\,,\theta\right)}\iota^*}{e^*{\left(\theta\right)}\iota}\right)}}\,\de \theta}+{\int\limits_0^\infty{\chi{\left(\theta\right)}\left(1-\xi{\left(\theta\right)}\right)a^*{\left(\theta\right)}{\left(1-\frac{a{\left(\,\cdot\,,\theta\right)}\iota^*}{a^*{\left(\theta\right)}\iota}\right)}}\,\de \theta}\right).
\end{align*}
In view of the above, we now write 
\begin{align*}
\dfrac{\de L}{\de t}&=-\mu S^*\left(\frac{S}{S^*}+\frac{S^*}{S}-2\right)-\left(\zeta\epsilon+\mu\right)V^*\left(\frac{V}{V^*}+\frac{S^*}{S}+\frac{SV^*}{S^*V}-3\right)+\\
&\phantom{=}+\left(S^*+\left(1-\epsilon\right)V^*\right){\int\limits_0^\infty{\left(\beta_A{\left(\theta\right)}a^*{\left(\theta\right)}+\beta_I{\left(\theta\right)}i^*{\left(\theta\right)}\right)\left(1-\frac{S^*}{S}+\ln{\frac{S^*}{S}}\right)}\,\de \theta}+\\
&\phantom{=}+S^*{\int\limits_0^\infty{\beta_A{\left(\theta\right)}a^*{\left(\theta\right)}\left(\ln{\frac{S\,a{\left(\,\cdot\,,\theta\right)}}{S^*a^*{\left(\theta\right)}}}-\ln{\frac{\alpha}{\alpha^*}}+1-\dfrac{S\,a{\left(\,\cdot\,,\theta\right)}\varepsilon^*}{S^*a^*{\left(\theta\right)}\varepsilon}\right)}\,\de \theta}+\\
&\phantom{=}+S^*{\int\limits_0^\infty{\beta_I{\left(\theta\right)}i^*{\left(\theta\right)}\left(\ln{\frac{S\,i{\left(\,\cdot\,,\theta\right)}}{S^*i^*{\left(\theta\right)}}}-\ln{\frac{\iota}{\iota^*}}+1-\dfrac{S\,i{\left(\,\cdot\,,\theta\right)}\varepsilon^*}{S^*i^*{\left(\theta\right)}\varepsilon}\right)}\,\de \theta}+\\
&\phantom{=}+\left(1-\epsilon\right)V^*{\int\limits_0^\infty{\beta_A{\left(\theta\right)}a^*{\left(\theta\right)}\left(1-\dfrac{SV^*}{S^*V}+\ln{\frac{S\,a{\left(\,\cdot\,,\theta\right)}}{S^*a^*{\left(\theta\right)}}}-\ln{\frac{\alpha}{\alpha^*}}+1-\dfrac{V\,a{\left(\,\cdot\,,\theta\right)}\varepsilon^*}{V^*a^*{\left(\theta\right)}\varepsilon}\right)}\,\de \theta}+\\
&\phantom{=}+\left(1-\epsilon\right)V^*{\int\limits_0^\infty{\beta_I{\left(\theta\right)}i^*{\left(\theta\right)}\left(1-\dfrac{SV^*}{S^*V}+\ln{\frac{S\,i{\left(\,\cdot\,,\theta\right)}}{S^*i^*{\left(\theta\right)}}}-\ln{\frac{\iota}{\iota^*}}+1-\dfrac{V\,i{\left(\,\cdot\,,\theta\right)}\varepsilon^*}{V^*i^*{\left(\theta\right)}\varepsilon}\right)}\,\de \theta}+\\
&\phantom{=}+\left(1-\epsilon\right)V^*{\int\limits_0^\infty{\beta_A{\left(\theta\right)}a^*{\left(\theta\right)}\left(\ln{\dfrac{SV^*}{S^*V}}-\ln{\dfrac{SV^*}{S^*V}}+\ln{\dfrac{V\,a{\left(\,\cdot\,,\theta\right)}\varepsilon^*}{V^*a^*{\left(\theta\right)}\varepsilon}}-\ln{\dfrac{V\,a{\left(\,\cdot\,,\theta\right)}\varepsilon^*}{V^*a^*{\left(\theta\right)}\varepsilon}}\right)}\,\de \theta}+\\
&\phantom{=}+\left(1-\epsilon\right)V^*{\int\limits_0^\infty{\beta_I{\left(\theta\right)}i^*{\left(\theta\right)}\left(\ln{\dfrac{SV^*}{S^*V}}-\ln{\dfrac{SV^*}{S^*V}}+\ln{\dfrac{V\,i{\left(\,\cdot\,,\theta\right)}\varepsilon^*}{V^*i^*{\left(\theta\right)}\varepsilon}}-\ln{\dfrac{V\,i{\left(\,\cdot\,,\theta\right)}\varepsilon^*}{V^*i^*{\left(\theta\right)}\varepsilon}}\right)}\,\de \theta}+\\
&\phantom{=}+f_A{\left(0\right)}{\int\limits_0^\infty{k{\left(\theta\right)}q{\left(\theta\right)}e^*{\left(\theta\right)}\left(\ln{\frac{e{\left(\,\cdot\,,\theta\right)}}{e^*{\left(\theta\right)}}}-\ln{\frac{\varepsilon}{\varepsilon^*}}\right)}\,\de \theta}+\\
&\phantom{=}+f_I{\left(0\right)}{\int\limits_0^\infty{k{\left(\theta\right)}\left(1-q{\left(\theta\right)}\right)e^*{\left(\theta\right)}\left(\ln{\frac{e{\left(\,\cdot\,,\theta\right)}}{e^*{\left(\theta\right)}}}-\ln{\frac{\varepsilon}{\varepsilon^*}}\right)}\,\de \theta}+\\
&\phantom{=}+f_I{\left(0\right)}{\int\limits_0^\infty{\chi{\left(\theta\right)}\left(1-\xi{\left(\theta\right)}\right)a^*{\left(\theta\right)}\left(\ln{\frac{a{\left(\,\cdot\,,\theta\right)}}{a^*{\left(\theta\right)}}}-\ln{\frac{\alpha}{\alpha^*}}\right)}\,\de \theta}+\\
&\phantom{=}+\left(S^*+\left(1-\epsilon\right)V^*\right){\int\limits_0^\infty{\beta_A{\left(\theta\right)}\e^{-\int\limits_0^\theta{\gamma_A{\left(s\right)}\xi{\left(s\right)}+\chi{\left(s\right)}\left(1-\xi{\left(s\right)}\right)+\mu\,\de s}}}\,\de \theta}{\int\limits_0^\infty{k{\left(\theta\right)}q{\left(\theta\right)}e^*{\left(\theta\right)}}\left(1-\frac{e{\left(\,\cdot\,,\theta\right)}\alpha^*}{e^*{\left(\theta\right)}\alpha}\right)\,\de \theta}+\\
&\phantom{=}+f_I{\left(0\right)}\left({\int\limits_0^\infty{k{\left(\theta\right)}\left(1-q{\left(\theta\right)}\right)e^*{\left(\theta\right)}{\left(1-\frac{e{\left(\,\cdot\,,\theta\right)}\iota^*}{e^*{\left(\theta\right)}\iota}\right)}}\,\de \theta}+{\int\limits_0^\infty{\chi{\left(\theta\right)}\left(1-\xi{\left(\theta\right)}\right)a^*{\left(\theta\right)}{\left(1-\frac{a{\left(\,\cdot\,,\theta\right)}\iota^*}{a^*{\left(\theta\right)}\iota}\right)}}\,\de \theta}\right). 
\end{align*}
\textit{Step IId:}\\
From the definition of $f$ and the equation 
\begin{align*}
D_1&=\dfrac{\alpha}{\alpha^*}\left(S^*+\left(1-\epsilon\right)V^*\right){\int\limits_0^\infty{\beta_A{\left(\theta\right)}\e^{-\int\limits_0^\theta{\gamma_A{\left(s\right)}\xi{\left(s\right)}+\chi{\left(s\right)}\left(1-\xi{\left(s\right)}\right)+\mu\,\de s}}}\,\de \theta}{\int\limits_0^\infty{k{\left(\theta\right)}q{\left(\theta\right)}e^*{\left(\theta\right)}}\,\de \theta}+\\
&\phantom{=}+\dfrac{\iota}{\iota^*}f_I{\left(0\right)}\left({\int\limits_0^\infty{k{\left(\theta\right)}\left(1-q{\left(\theta\right)}\right)e^*{\left(\theta\right)}}\,\de \theta}+{\int\limits_0^\infty{\chi{\left(\theta\right)}\left(1-\xi{\left(\theta\right)}\right)a^*{\left(\theta\right)}}\,\de \theta}\right),
\end{align*}
the expression can eventually be simplified as follows 
\begin{align*}
\dfrac{\de L}{\de t}&=-\mu S^*\left(\frac{S}{S^*}+\frac{S^*}{S}-2\right)-\left(\zeta\epsilon+\mu\right)V^*\left(\frac{V}{V^*}+\frac{S^*}{S}+\frac{SV^*}{S^*V}-3\right)-\\
&\phantom{=}-\left(S^*+\left(1-\epsilon\right)V^*\right){\int\limits_0^\infty{\left(\beta_A{\left(\theta\right)}a^*{\left(\theta\right)}+\beta_I{\left(\theta\right)}i^*{\left(\theta\right)}\right)f{\left(\frac{S^*}{S}\right)}}\,\de \theta}-\\
&\phantom{=}-S^*{\int\limits_0^\infty{\beta_A{\left(\theta\right)}a^*{\left(\theta\right)}f{\left(\dfrac{S\,a{\left(\,\cdot\,,\theta\right)}\varepsilon^*}{S^*a^*{\left(\theta\right)}\varepsilon}\right)}}\,\de \theta}-S^*{\int\limits_0^\infty{\beta_I{\left(\theta\right)}i^*{\left(\theta\right)}f{\left(\dfrac{S\,i{\left(\,\cdot\,,\theta\right)}\varepsilon^*}{S^*i^*{\left(\theta\right)}\varepsilon}\right)}}\,\de \theta}-\\
&\phantom{=}-\left(1-\epsilon\right)V^*{\int\limits_0^\infty{\beta_A{\left(\theta\right)}a^*{\left(\theta\right)}f{\left(\dfrac{V\,a{\left(\,\cdot\,,\theta\right)}\varepsilon^*}{V^*a^*{\left(\theta\right)}\varepsilon}\right)}}\,\de \theta}-\left(1-\epsilon\right)V^*{\int\limits_0^\infty{\beta_I{\left(\theta\right)}i^*{\left(\theta\right)}f{\left(\dfrac{V\,i{\left(\,\cdot\,,\theta\right)}\varepsilon^*}{V^*i^*{\left(\theta\right)}\varepsilon}\right)}}\,\de \theta}-\\
&\phantom{=}-\left(1-\epsilon\right)V^*{\int\limits_0^\infty{\left(\beta_A{\left(\theta\right)}a^*{\left(\theta\right)}+\beta_I{\left(\theta\right)}i^*{\left(\theta\right)}\right)f{\left(\dfrac{SV^*}{S^*V}\right)}}\,\de \theta}-\\
&\phantom{=}-\left(S^*+\left(1-\epsilon\right)V^*\right){\int\limits_0^\infty{\beta_A{\left(\theta\right)}\e^{-\int\limits_0^\theta{\gamma_A{\left(s\right)}\xi{\left(s\right)}+\chi{\left(s\right)}\left(1-\xi{\left(s\right)}\right)+\mu\,\de s}}}\,\de \theta}{\int\limits_0^\infty{k{\left(\theta\right)}q{\left(\theta\right)}e^*{\left(\theta\right)}}f{\left(\frac{e{\left(\,\cdot\,,\theta\right)}\alpha^*}{e^*{\left(\theta\right)}\alpha}\right)}\,\de \theta}-\\
&\phantom{=}-f_I{\left(0\right)}{\int\limits_0^\infty{k{\left(\theta\right)}\left(1-q{\left(\theta\right)}\right)e^*{\left(\theta\right)}}f{\left(\frac{e{\left(\,\cdot\,,\theta\right)}\iota^*}{e^*{\left(\theta\right)}\iota}\right)}\,\de \theta}-f_I{\left(0\right)}{\int\limits_0^\infty{\chi{\left(\theta\right)}\left(1-\xi{\left(\theta\right)}\right)a^*{\left(\theta\right)}}f{\left(\frac{a{\left(\,\cdot\,,\theta\right)}\iota^*}{a^*{\left(\theta\right)}\iota}\right)}\,\de \theta}.
\end{align*}
\textit{Step III:}\\
Employing the arithmetic-geometric mean inequality, we get $$\mathcal{R}_0\leq 1\Rightarrow\,\frac{\de L}{\de t}\leq 0,\text{ }\forall t\in\mathbb{R}_0^+$$ and the equality holds only for the endemic steady state, i.e. when 
$$\left(S,V,e,a,i\right)=\left(S^*,V^*,e^*,a^*,i^*\right).$$ 
Hence, the singleton $\left\{\left(S^*,V^*,e^*,a^*,i^*\right)\right\}$ is the largest invariant set for which $$\frac{\de L}{\de t}=0.$$ Then, from the LaSalle in-variance principle it follows that the endemic steady state is globally asymptotically stable.
\end{proof}

\section{Numerical simulations}
\label{sec:numerics}

In this section, we numerically solve $\mathscr{P}$ \eqref{SVEIAR-age-scl} in order to verify the validity of the analysis performed in \hyperref[sec:derivanal]{\S \ref*{sec:derivanal}} and to further investigate the behavior of $\mathscr{P}$ \eqref{SVEIAR-age-scl}.

\subsection{Numerical scheme}
Here, we present the temporal discretization used to numerically solve $\mathscr{P}$ and the code used to implement it.

\subsubsection{Temporal discretization}

We assume that the maximum age of the population, $\theta_{\dagger}$, is equal to $90 \cdot 360$ days. Furthermore, we study $\mathscr{P}$ \eqref{SVEIAR-age-scl} for a time of up to 1500 days. Hence, we solve $\mathscr{P}$ \eqref{SVEIAR-age-scl} in the interval $\left(t, \theta \right) \in \left[0,\; 1500 \right] \times \left[0,\; 90 \cdot 360 \right] \; \cdot $ days. The time-age step we chose is $h = 0.05$. Let $\mathcal{N}$ be the number of time-age steps needed to reach the maximum age, i.e $\theta_\dagger$, and $\mathcal{J}$ be the number of time-age steps needed to reach the maximum time, i.e 1500 days.

To discretize the time derivative, we use the following first-order forward difference scheme:
\begin{equation*}
    \frac\partial{\partial t}\left( u(t^n) \right) = \lim_{h\to0^+} \frac{u(t^n+h)-u(t^n)}h \approx\frac{u^{n+1}-u^n}h, \quad 0 \le n \le \mathcal{N}-1 \;,
\end{equation*}
for $u\in\,\left\{S(t), \; V(t) \,\big|\, t \in  \left[0,\; 1500 \right]\; \cdot \text{ days} \right\}$.

To discretize the temporal directional derivative, we use the following first-order approximation:
\begin{equation*}
   \left( \frac\partial{\partial t}+\frac\partial{\partial \theta} \right) \left( u(t^n, \theta_j) \right) =\lim_{h\to0^+}\frac{u(t^n+h,\theta_j+h)-u(t^n,\theta_j)}h\approx\frac{u_{j+1}^{n+1}-u_j^n}h, \quad 0 \le n \le \mathcal{N}-1, \; 0 \le j \le \mathcal{J}-1 \;,
\end{equation*}
for $u\in\,\left\{e{\left(t,\,\theta\,\right)},a{\left(t,\,\theta\,\right)},i{\left(t,\,\theta\,\right)}\,\big|\,\left(t,\; \theta \right)\in  \left[0,\; 1500 \right] \times \left[0,\; 90 \cdot 360 \right] \; \cdot \text{ days} \right\} \;.$

To discretize the integrals, we use the following quadrature formula:
\begin{equation*}
    \int\limits_0^\infty{g{\left(\theta\right)}u{\left(\,t^n\,,\theta\right)}\,\de\theta} \approx h\sum\limits_{j=0}^{\mathcal{J}}{g{\left(\theta_j\right)}u{\left(\,t^n\,,\theta_j\right)}} = 
    h\sum\limits_{j=0}^{\mathcal{J}}{g_j u^n_j} , \quad 0 \le n \le \mathcal{N}-1 \;,
\end{equation*}
for 
\begin{multline*}
    \left(u,g\right)\in\,\left\{e{\left(t,\,\theta\,\right)},a{\left(t,\,\theta\,\right)},i{\left(t,\,\theta\,\right)}\,\big|\,\left(t,\; \theta \right)\in  \left[0,\; 1500 \right] \times \left[0,\; 90 \cdot 360 \right] \; \cdot \text{ days} \right\}\times \\
   \times \left\{\beta_A{\left(\theta\right)},\beta_I{\left(\theta\right)},k{\left(\theta\right)},q{\left(\theta\right)},\gamma_A{\left(\theta\right)},\xi{\left(\theta\right)},\chi{\left(\theta\right)},\gamma_I{\left(\theta\right)} \,\big|\, \theta \in  \left[0,\; 90 \cdot 360 \right] \; \cdot \text{ days}\right\} \;.
\end{multline*}

\subsubsection{Code implementation}

To implement the aforementioned discretization schemes, we use \pkg{Julia (v1.8.5)} \cite{Julia}. The code can be found at \url{https://github.com/TsilidisV/age-structured-SVeaiR-model}. To plot the numerical solution of $\mathscr{P}$ \eqref{SVEIAR-age-scl}, we use \pkg{Makie.jl} \cite{DanischKrumbiegel2021}. To save and load the results, we use \pkg{JLD2.jl} and \pkg{CodecZlib.jl}. To calculate $\mathcal{R}_0$, we use \pkg{QuadGK.jl} \cite{quadgk} and \pkg{Integrals.jl} \cite{DifferentialEquations.jl-2017}. To create faster \pkg{Julia} structs for the parameters and initial conditions, we use \pkg{FunctionWrappers.jl}. Finally, we use \pkg{Dierckx.jl} to interpolate, as well as \pkg{CSV.jl} and \pkg{DataFrames.jl} \cite{dataframes.jl} to load the data for the parameter values.

\subsection{Parameter values}
Here, we give a description of the parameter values chosen to represent the case of SARS-CoV-2.  A summary of the parameter values, can be found in \hyperref[tab:paramValues]{Table \ref*{tab:paramValues}}. 

\begin{table}[ht]
\caption{A list of parameters of $\mathscr{P}$ \eqref{SVEIAR-age-scl}, along with their value, units and value source}
\centering
\begin{tabular}{@{}cccc@{}}
\toprule
Parameters & Value                   & Units                         & Source \\ \midrule
$N_0$      & $80 \cdot 10^6$           & individuals                   & Estimated from \cite{ourworldindata}       \\
$\mu$      & $4.38356 \cdot 10^{-5}$ & day $^{-1}$                   & Estimated from \cite{ourworldindata}       \\
$\beta_A$  &  \hyperref[fig:2]{Figure \ref*{fig:2}}                      & individual$^{-1}\, \cdot \, $day $^{-1}$ & Estimated from \cite{DelValle_Hyman_Hethcote_Eubank_2007}       \\
$\beta_I$  &  \hyperref[fig:2]{Figure \ref*{fig:2}}                       & individual$^{-1}\, \cdot \, $day$^{-1}$ & Estimated from \cite{DelValle_Hyman_Hethcote_Eubank_2007}       \\
$p$        & $10^{-3}$               & day$^{-1}$                    & Estimated from \cite{ourworldindata}      \\
$\epsilon$ & 0.7                     & -                             & Estimated from \cite{grant2022impact}      \\
$\zeta$    & $\frac{1}{14}$                    & day$^{-1}$                    & Estimated from \cite{chau2022immunogenicity}      \\
$k$        &  \hyperref[eq:valuek]{Equation \ref*{eq:valuek}}                        & day$^{-1}$                    &  Estimated from \cite{kang2022transmission,wu2022incubation}      \\
$q$        &   \hyperref[fig:3]{Figure \ref*{fig:3}}                       & -                             &  Estimated from \cite{sah2021asymptomatic}      \\
$\xi$      & 0.5                     & -                             & Estimated from \cite{he2021proportion,buitrago2022occurrence}       \\
$\chi$     & \hyperref[eq:valuechi]{Equation \ref*{eq:valuechi}}                         & day$^{-1}$                    & Estimated from \cite{he2021proportion,buitrago2022occurrence}       \\
$\gamma_A$ & $\frac{1}{8}$                    & day$^{-1}$                    & Estimated from \cite{byrne2020inferred}      \\
$\gamma_I$ & $\frac{1}{14}$                    & day$^{-1}$                    & Estimated from \cite{byrne2020inferred}      \\ \bottomrule
\end{tabular}
\label{tab:paramValues}
\end{table}

\begin{itemize}
    \item $N_0 = 80 \cdot 10^6$, the size of the population, is chosen to be that of a relative large country \cite{ourworldindata}.
    \item $\mu = 4.38356 \cdot 10^{-5} \; \text{day}^{-1}$, the birth/death rate, is converted from the average birth/death rate of the world for the year 2021, 16 per 1000 individuals per year, found in \cite{ourworldindata}.
    \item $\beta_A$ and $\beta_I$ are functions of age and are estimated from \cite{DelValle_Hyman_Hethcote_Eubank_2007}. As can be seen from Fig. 2 of \cite{DelValle_Hyman_Hethcote_Eubank_2007}, the average contacts an individual makes each day regardless of their epidemiological status is about 16.71 contacts per day. In order to examine the effect of age in the dynamics of $\mathscr{P}$ \eqref{SVEIAR-age-scl}, we assume the following two functions to
respectively model two extreme cases of the average number of contacts an individual makes:
    \begin{subequations} \label{eq:contac_funcs}
        \begin{align}
                c_1(\theta) &= \frac{16.71}{0.38}\exp\left(\left(\frac{\theta-80\omega}{10^4}\right)^2\right),\quad \theta \in \left[0,\; 90 \cdot 360 \right] \; \cdot \; \text{days} \label{eq:contac_funcs1}\\
                c_2(\theta) &= \frac{16.71}{0.38}\exp\left(\left(\frac{\theta-10\omega}{10^4}\right)^2\right),\quad \theta \in \left[0,\; 90 \cdot 360 \right] \; \cdot \; \text{days} \;.\label{eq:contac_funcs2}
        \end{align}

    \end{subequations}
    Both $c_1$ and $c_2$ have the same mean value of 16.71 contacts per day in the interval $\left[0,\; 90 \cdot 360 \right] \; \cdot \; \text{days}$. We additionally assume that the probability of an exposed individual passing to the compartments of asymptomatic and symptomatic individuals to be $\varpi_{E\to A} = \frac{1}{5}$ and $\varpi_{E\to I} = \frac{2}{5}$, respectively. Finally, assuming the transmission rates to be defined as $\beta_{A_i} = \frac{c_i \cdot \varpi_{E\to A}}{N_0}$ and $ \beta_{I_i}= \frac{c_i \cdot \varpi_{E\to I}}{N_0},$ for $i = 1,2$, we get \hyperref[fig:2]{Figure \ref*{fig:2}}.

    \begin{figure}[H]
    \centering
    \includegraphics[width=1\textwidth]{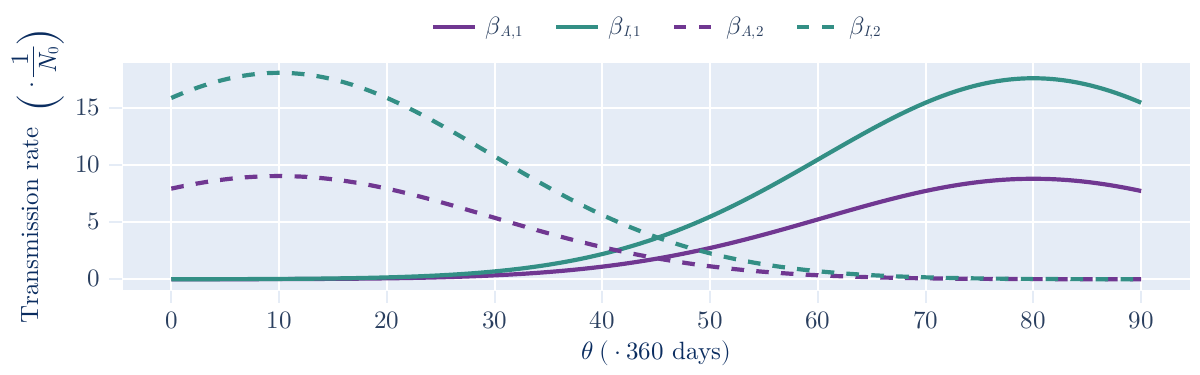}
    \caption{Two extreme types of transmission rates for the symptomatic infectious and asymptomatic infectious individuals. The transmission rates corresponding to the contact function of $c_1$ peak at individuals of 10 years of age, whereas the transmission rates corresponding to the contact function of $c_2$ peak at individuals of 80 years of age.}
    \label{fig:2}
    \end{figure} 
    
    \item $p = 10^{-3} \; \text{day}^{-1}$, the vaccination rate, is assumed to be that during the summer of 2021 in the USA \cite{ourworldindata}.
    \item $\epsilon = 0.7$, the vaccine effectiveness, is assumed to be an average effectiveness of the BNT162b2 and ChAdOx1 nCoV-19 vaccine \cite{grant2022impact}.
    \item $\zeta = \frac{1}{14}$, the vaccine-induced immunity rate, is taken from \cite{chau2022immunogenicity}.
    \item $k$, the latent rate, is a function of age and is taken by assuming that the latent and incubation period differ by one day \cite{kang2022transmission,wu2022incubation}. It is given by \begin{equation} \label{eq:valuek}
                k(\theta) = \begin{cases}
                \frac{1}{4}   \; \text{day}^{-1},   & \theta < 30 \cdot 360 \\
                \frac{1}{4.8} \; \text{day}^{-1},   & 30 \cdot 360 \le \theta < 40 \cdot 360 \\
                \frac{1}{4.8} \; \text{day}^{-1},   & 40 \cdot 360 \le \theta < 50 \cdot 360\\
                \frac{1}{5.5} \; \text{day}^{-1},   & 50 \cdot 360 \le \theta < 60 \cdot 360\\
                \frac{1}{3.1} \; \text{day}^{-1},   & 60 \cdot 360 \le \theta < 70 \cdot 360\\
                \frac{1}{6}   \; \text{day}^{-1},   & 70 \cdot 360 \le \theta \; .
                \end{cases}
         \end{equation}
    \item $q$, the proportion of the latent/exposed individuals becoming asymptomatic
infectious is taken from \cite{sah2021asymptomatic} and can be seen in \hyperref[fig:3]{Figure \ref*{fig:3}}. To digitise the data from \cite{sah2021asymptomatic}, we use \pkg{WebPlotDigitizer 4.6} \cite{Rohatgi2022}. 
    \begin{figure}[H]
    \centering
    \includegraphics[width=1\textwidth]{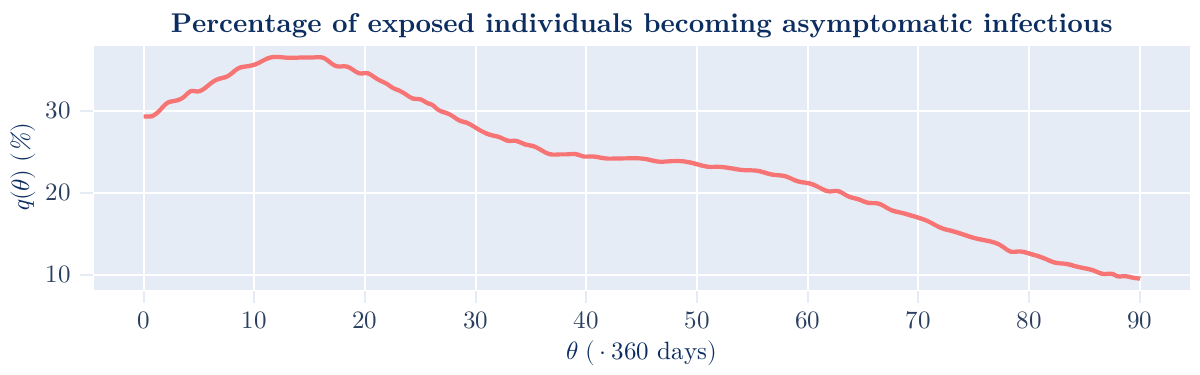}
    \caption{Two types of transmission rates for the symptomatic infectious and asymptomatic infectious individuals. The transmission rates corresponding to the contact function of $c_1$ peak at individuals of 10 years of age, whereas the transmission rates corresponding to the contact function of $c_2$ peak at individuals of 80 years of age.}
    \label{fig:3}
    \end{figure} 

    \item $\xi = 0.5$, the proportion of the asymptomatic infectious individuals becoming recovered/removed without developing any symptoms, is estimated from \cite{he2021proportion,buitrago2022occurrence}.
    \item $\chi$, the incubation rate, is a function of age and is taken form data from \cite{tan2020does}. It is given by
    \begin{equation}  \label{eq:valuechi}
    \chi(\theta) = \begin{cases}
    \frac{1}{5}   \; \text{day}^{-1},   & \theta < 30 \cdot 360\\
    \frac{1}{5.8} \; \text{day}^{-1},   & 30 \cdot 360 \le \theta < 40 \cdot 360\\
    \frac{1}{5.8} \; \text{day}^{-1},   & 40 \cdot 360 \le \theta < 50 \cdot 360\\
    \frac{1}{6.5} \; \text{day}^{-1},   & 50 \cdot 360 \le \theta < 60 \cdot 360\\
    \frac{1}{4.1} \; \text{day}^{-1},   & 60 \cdot 360 \le \theta < 70 \cdot 360\\
    \frac{1}{7}   \; \text{day}^{-1},   & 70 \cdot 360 \le \theta \; .
    \end{cases}
    \end{equation}
     \item $\gamma_A = \frac{1}{8} \; \text{day}^{-1}$, the recovery rate of the asymptomatic infectious individuals, is a function of age, but it is taken as a constant due to lack of available data. It is estimated from \cite{byrne2020inferred}.
     \item $\gamma_I = \frac{1}{14} \; \text{day}^{-1}$, the recovery rate of the symptomatic infectious individuals, is a function of age, but it is taken as a constant due to lack of available data. It is estimated from \cite{byrne2020inferred}.
\end{itemize}

\subsection{Results}

Throughout our simulations we assume that $S_0=V_0=2\cdot 10^7$ individuals. In order to study $\mathscr{P}$ \eqref{SVEIAR-age-scl} in a global scale, we vary the rest of the initial conditions. In particular, we assume that $E_0 = A_0 = I_0 = d$ and let $d$ take the values of $10, \; 10^4, \; 10^6, \; 4 \cdot 10^6, \; 10^7$.

\subsubsection{The case of \texorpdfstring{$\mathcal{R}_0 \leq 1$}{R0<=1}} 

Here, we assume the average number of contacts of each individual, $c$, to be as in \eqref{eq:contac_funcs1}, i.e. $c=c_1$. In such a case, $\mathcal{R}_0=5.95 \cdot 10^{-5}$. As we see in \hyperref[fig:figure4]{Figure \ref*{fig:figure4}}, for every initial condition we have that $\left( E, \; A,\; I\right) \to \left( 0, \; 0,\; 0\right)$, as $t \to \infty$. This confirms the global stability analysis performed in \hyperref[sec:derivanal]{\S \ref*{sec:derivanal}}, since the solutions converge to the disease-free steady state for every initial condition when $\mathcal{R}_0 \leq 1$.

\begin{figure}[H]
\centering
\includegraphics[width=1\textwidth]{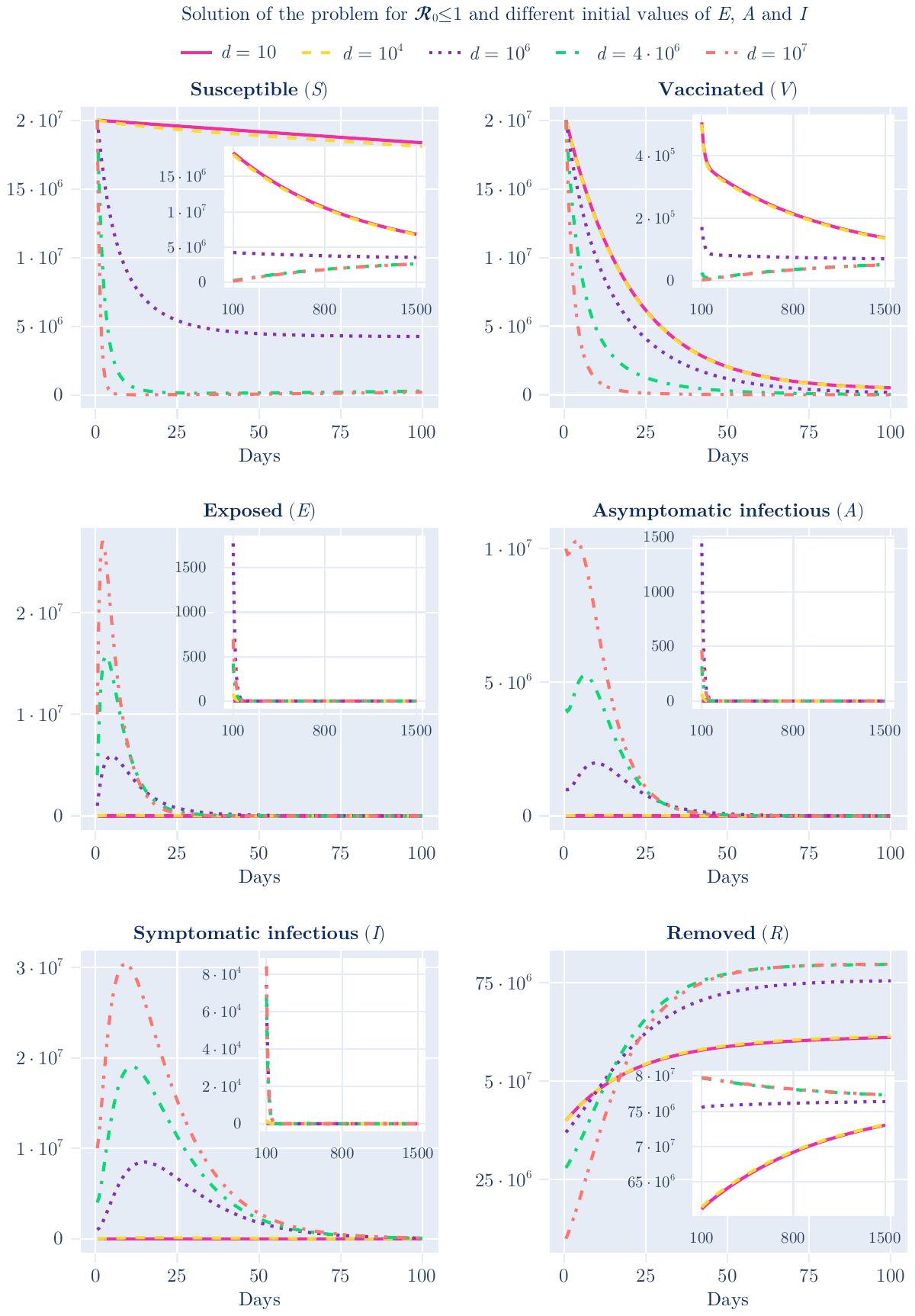}
\caption{Solution of $\mathscr{P}$ \eqref{SVEIAR-age-scl} for $t \in \left[0, 1500 \right] \cdot $ days. The time in each large diagram is in the range $\left[0, 100 \right] \cdot $ days, whereas in each inserted small diagram in the range of $\left[100, 1500 \right] \cdot $ days. The parameter values are as in \hyperref[tab:paramValues]{Table \ref*{tab:paramValues}}, with $c = c_1$, and the initial conditions for $S$ and $V$ are  $S_0=V_0=2\cdot 10^7$. The time-related initial conditions $E_0, \; A_0$ and $I_0$ for $E, \; A$ and $I$, respectively, are all equal, i.e $E_0 = A_0 = I_0 = d$. The value of $d$ takes the values of $10, 10^4, 10^6, 4 \cdot 10^6$ and $10^7$. We see that for all initial conditions the solutions converge towards the disease-free steady state, since $\left( E, \; A,\; I\right) \to \left( 0, \; 0,\; 0\right)$, as $t \to \infty$. Hence, the global stability of the disease-free steady state for $\mathcal{R}_0 \le 1$ is numerically confirmed.}
\label{fig:figure4}
\end{figure}

\subsubsection{The case of \texorpdfstring{$\mathcal{R}_0>1$}{R0>1}} 

Here, we assume the average number of contacts of each individual, $c$, to be as in \eqref{eq:contac_funcs2}, i.e. $c=c_2$. In such a case, $\mathcal{R}_0=9.14$. As we see in \hyperref[fig:figure5]{Figure \ref*{fig:figure5}}, for every initial condition we have that  $\left(E,\;A,\;I\right)$ converges to a nonzero value, as $t \to \infty$. This confirms the global stability analysis performed in \hyperref[sec:derivanal]{\S \ref*{sec:derivanal}}, since the solutions converge, in an oscillatory way, to the endemic steady state for every initial condition when $\mathcal{R}_0>1$.

\begin{figure}[H]
\centering
\includegraphics[width=1\textwidth]{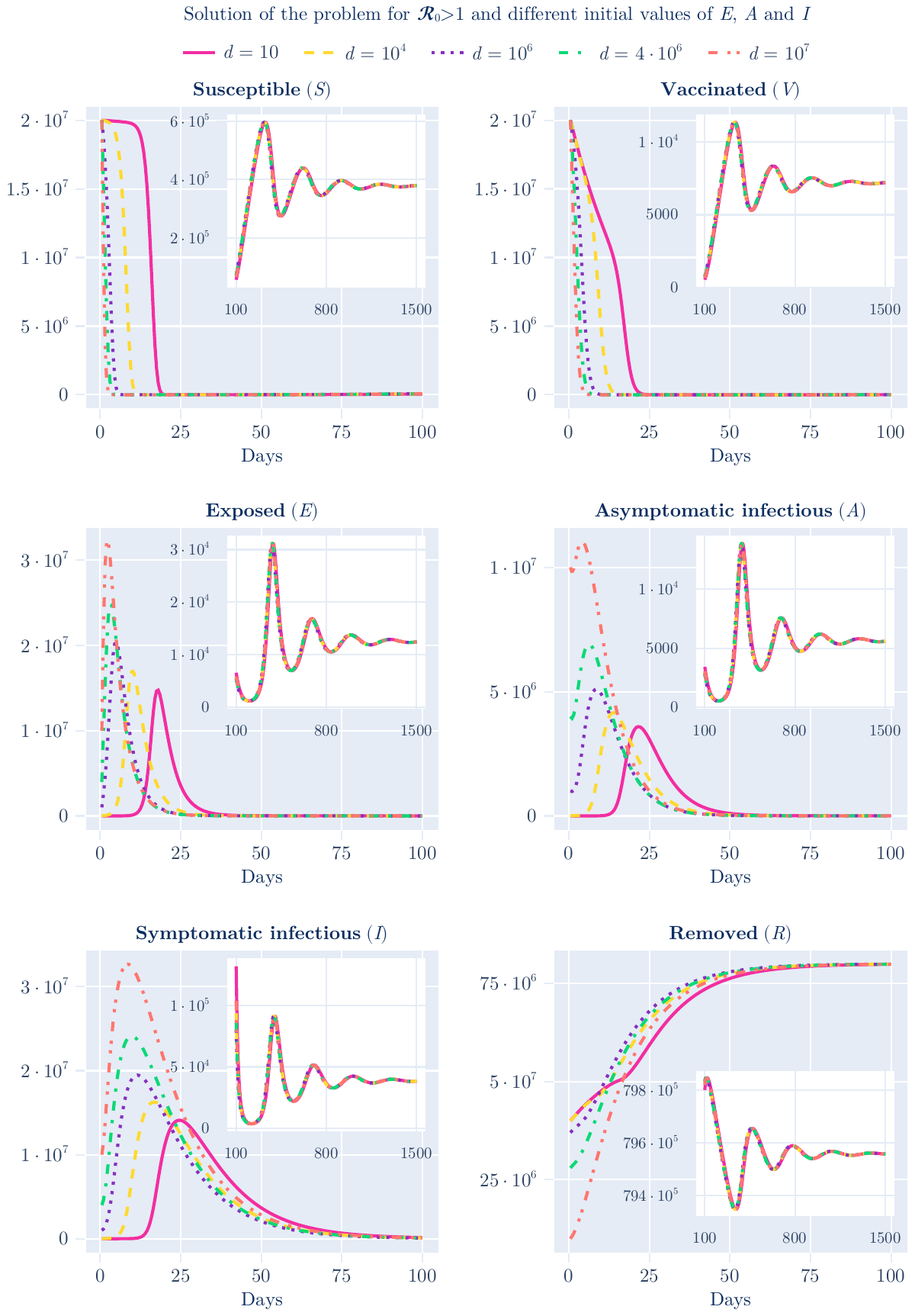}
\caption{Solution of $\mathscr{P}$ \eqref{SVEIAR-age-scl} for $t \in \left[0, 1500 \right] \cdot $ days. The time in each large diagram is in the range $\left[0, 100 \right] \cdot $ days, whereas in each inserted small diagram in the range of $\left[100, 1500 \right] \cdot $ days. The parameter values are as in \hyperref[tab:paramValues]{Table \ref*{tab:paramValues}}, with $c=c_2$, and the initial conditions for $S$ and $V$ are  $S_0=V_0=2\cdot 10^7$. The time-related initial conditions $E_0, \; A_0$ and $I_0$ for $E, \; A$ and $I$, respectively, are all equal, i.e $E_0 = A_0 = I_0 = d$. The value of $d$ takes the values of $10,\; 10^4,\; 10^6, 4 \cdot 10^6$ and $10^7$. We see that for all initial conditions the solutions converge towards the endemic steady state, since $\left(E,\;A,\;I\right)$ converges to a nonzero value, as $t \to \infty$. Interestingly, the convergence to the endemic steady state is oscillatory. Hence, the global stability of the endemic steady state for $\mathcal{R}_0 > 1$ is numerically confirmed.}
\label{fig:figure5}
\end{figure} 


\section{Conclusions and discussion}
\label{sec:CD}
In this paper, we derived an age-structured epidemiological compartment problem 
and we studied it in terms of global well-posedness and stability analysis. From this analysis we deduced the basic reproductive number, $\mathcal{R}_0$, of the model, a critical measurement of the transmission potential of a disease. 

The model presented in this paper focused on the age structure of a population. A straightforward generalization includes the consideration of more independent variables, such as a spatial one. Moreover, it would be essential to include additional, potentially important factors of the evolution of the epidemiological phenomenon, such as waning immunity gained by both infected and vaccinated individuals.

 \bibliographystyle{plain}
\bibliography{mybibfile}\label{bibliography}
\addcontentsline{toc}{chapter}{Bibliography}

\end{document}